\documentclass{article}
\usepackage{amssymb}
\usepackage{amsmath}
\usepackage{amsthm}
\usepackage{mathdots}
\usepackage{xypic}
\usepackage{url}

\newtheorem{thm}{Theorem}[section]
\newtheorem{lem}[thm]{Lemma}

\newtheorem{prop}[thm]{Proposition}
\newtheoremstyle{mydef}
{\topsep}{\topsep}%
{}{}%
{\bfseries}{}
{\newline}
{%
  \thmname{#1}~\thmnumber{#2}\thmnote{\ -\ #3}.\\*[-1.5ex]%
}%

\theoremstyle{mydef}
\newtheorem{definition}[thm]{Definition}

\theoremstyle{definition}
\newtheorem{exmp}[thm]{Example}

\title{Twisted endoscopy from a sheaf-theoretic perspective}
\author{Aaron Christie and Paul Mezo}
\date{}

\begin{document}

\maketitle

\begin{abstract}
The standard theory of endoscopy for real groups has two parallel
formulations.  The original formulation of Langlands and Shelstad relies on
methods in harmonic analysis. The subsequent formulation of Adams,
Barbasch and Vogan relies on sheaf-theoretic methods.  The original
formulation was extended by Kottwitz and Shelstad to twisted
endoscopy. We extend the sheaf-theoretic formulation to the context of
twisted endoscopy and provide applications for computing Arthur packets.
\end{abstract}
\section{Introduction}

Let $G$ be a connected reductive algebraic group defined over
a local field $F$ and let $G(F)$ be its group of $F$-points.  One may
view the theory of endoscopy as an 
endeavour to make precise connections between the representations of
$G(F)$ and the representations of $H(F)$, where $H$ is a reductive
group which is, in some sense,  inside $G$.  The theory is
largely conjectural over nonarchimedean fields, but over archimedean fields
a great deal has been worked out.

The groundbreaking results in the archimedean case are due to
Langlands and Shelstad  (\cite{langclass}, 
\cite{shelstad}). Their results were achieved largely by means of
harmonic analysis---the precise connections between the
representations of $G(F)$ and $H(F)$ that they proved are expressed
through identities between distribution characters. An alternative
perspective on endoscopy that does not derive its identities from
harmonic analysis was developed by Adams, Barbasch and Vogan,
culminating in the book  \cite{abv}. There, endoscopic identities are
obtained using a duality between representations and constructible (or
perverse) sheaves on varieties attached to $G$ and $H$. 

A natural question when presented with these two theories is to what
extent the identities of one theory can be obtained from the
identities of the other. Further, one could ask this question in the
setting of \emph{twisted} rather than standard endoscopy. That is
exactly the line of inquiry we pursue here. Before describing the
details of paper, let us give a brief tour of the ideas mentioned
above.

We take for granted that the reader is somewhat familiar with the original
perspective on endoscopy.  For simplicity, we assume here that $G$ is
adjoint and quasisplit over $\mathbb{R}$.  An \emph{L-parameter} is an
L-homomorphism
\begin{equation}
  \label{lhom}
  \phi_{G}: W_{\mathbb{R}} \rightarrow {^L}G
\end{equation}
of the real Weil group into the Langlands dual group.  The local
Langlands Correspondence pairs $\phi_{G}$ with an \emph{L-packet}
$\Pi_{\phi_{G}}$, a finite set of (infinitesimal equivalence classes
of) irreducible admissible representations of $G(\mathbb{R})$.  If
$\phi_{G}$ is an L-parameter corresponding to tempered
representations, then the correspondence may be refined.  In this case
Shelstad has shown that there is an injective map
\begin{equation}
  \label{injchar}
  \pi \in \Pi_{\phi_{G}} \mapsto \tau_{\phi_{G}}(\pi)
  \end{equation}
to  the irreducible characters of the component group of
${^\vee}G_{\phi_{G}}$, the centralizer in the dual  group ${^\vee}G$ of the 
image of $\phi_{G}$ (Corollary 11.1 \cite{she3}). 

A standard \emph{endoscopic group} $H$ may be obtained by choosing a semisimple
element $s \in {^\vee}G_{\phi_{G}}$  and setting the dual group
${^\vee}H$ to equal the identity component of the fixed-point subgroup
${^\vee}G^{\mathrm{Int}(s)}$ of the inner automorphism
$\mathrm{Int}(s)$.  By definition, $H$ is a quasisplit reductive
group. 
Avoiding some technicalities, we assume  the
inclusions $\phi_{G}(W_{\mathbb{R}}) \subset 
{^L}H \subset {^L}G$.  We may then define the L-parameter $\phi_{H}:
W_{\mathbb{R}}   \rightarrow {^L}H$ by setting its values equal to
those of $\phi_{G}$.  In the tempered case, Shelstad's endoscopic
identities have the form
\begin{equation}
  \label{shelend}
\mathrm{Lift} \left( \sum_{\pi_{H} \in \Pi_{\phi_{H}}}
\Theta_{\pi_{H}}\right) = \sum_{\pi \in \Pi_{\phi_{G}}}
\tau_{\phi_{G}}(\pi)(\dot{s})\  \Theta_{\pi}
\end{equation}
(Corollary 11.7 \cite{she3}).  On the right, $\Theta_{\pi}$ is the
distribution character of $\pi \in \Pi_{\phi_{G}}$, $\tau_{\phi_{G}}(\pi)$ is
the character of (\ref{injchar}), and $\dot{s}$ is the image of $s$ in
the component group.  On the left, $\mathrm{Lift}$ is defined through a
map of orbital integrals (Theorem 6.2 \cite{she3}).  It is significant
that the distribution being lifted on the left is \emph{stable} (3
\cite{arthur89}).

There are two different ways in which one would like to extend
(\ref{shelend}).  First, one would like to include nontempered
representations  while maintaining stability.  To this end, Arthur
has extended the notion of L-parameters
to so-called \emph{A-parameters},  and conjectured an extension of
(\ref{shelend}) for \emph{A-packets} (4 \cite{arthur89}).

The other desired extension of (\ref{shelend}) is to twisted endoscopy.  The
basic idea of twisted endoscopy is to allow endoscopic groups to be
defined through outer automorphisms.  Specifically, one would like to
define $H$ by allowing ${^\vee}H$ to equal the identity component of
the fixed-point subgroup  ${^\vee}G^{{^\vee}\vartheta}$ of an outer automorphism
${^\vee}\vartheta$ of ${^\vee}G$.  The foundations for twisted
endoscopy were laid by Kottwitz and Shelstad in \cite{ks}, and a
partial extension of (\ref{shelend}) has been proved for real groups
(\cite{sheltwist}, \cite{mezotwist}). 

Recently, Arthur knit together these two extensions of (\ref{shelend})
in a global context and consequently proved the nontempered version of
(\ref{shelend}) for orthogonal and symplectic groups (Theorem 2.2.1
\cite{arthurbook}). His A-packets are defined using twisted endoscopy
for $G = \mathrm{GL}_{N}$.

We now juxtapose the theory of Adams, Barbasch and Vogan with this
terse reminder of endoscopy.  The underlying novelty in their
theory is the introduction of a topological space $X({^\vee}G^{\Gamma})$ which
reparametrizes the set of L-homomorphisms (\ref{lhom}).  This space
is equipped with a ${^\vee}G$-action and it is the ${^\vee}G$-orbits
which are in bijection with the equivalence classes of
L-homomorphisms.  A remarkable property of $X({^\vee}G^{\Gamma})$ is
that the closure relations of its ${^\vee}G$-orbits imply 
relationships between the representations of different L-packets
(Proposition 1.11 \cite{abv}).

The characters $\tau_{\phi_{G}}(\pi)$ of (\ref{injchar}) are subsumed
so that the local Langlands Correspondence becomes a \emph{bijection}
between these characters and (infinitesimal equivalence classes of) irreducible
representations of an extended L-packet (Theorem 1.18 \cite{abv}).
The L-packets are extended in that they include the usual L-packets of all
inner forms of $G(\mathbb{R})$ (actually, all  \emph{strong real
  forms} as in
Definition 1.13 \cite{abv}).  Putting these details aside,  the
extended Langlands Correspondence of \cite{abv} allows one to match an
irreducible representation $\pi$ of $G(\mathbb{R})$ with a
${^\vee}G$-orbit $S_{\phi_{G}} \subset X({^\vee}G^{\Gamma})$ and a character
$\tau_{\phi_{G}}(\pi)$.  From the pair
$(S_{\phi_{G}},\tau_{\phi_{G}}(\pi))$ one may in turn define a local
system on $S_{\phi_{G}}$ and a constructible sheaf on
$X({^\vee}G^{\Gamma})$.  Thus, we have a geometric  interpretation of
representations.  By fine-tuning this geometric interpretation to the
implications of the closure relations among the ${^\vee}G$-orbits, a duality is
defined between representations and sheaves (Theorem 1.24 \cite{abv}).  

All of this may be carried out for an endoscopic group $H$, and so we
may reconsider the tempered endoscopic identity (\ref{shelend}) from
this new viewpoint.  The analogue 
of (\ref{shelend}) is Proposition 26.7 \cite{abv}.   In this
proposition the right-hand
side differs  by taking the sum over an \emph{extended} L-packet.
Similarly, the sum on the left-hand side of (\ref{shelend}) is taken
over an extended L-packet.  However, the lift of the left-hand side is
not defined through a map of orbital integrals.  Instead, it is
defined through duality together with the restriction functor on constructible
sheaves.

A nontempered analogue of (\ref{shelend}) is also proven in Theorem
26.24 \cite{abv}.  The theory surrounding
this nontempered analogue, that of \emph{microlocal geometry}, is
quite sophisticated.  We will limit ourselves to saying only that the
sums in the nontempered version of (\ref{shelend}) are taken over
\emph{micro-packets} rather than extended L-packets.
The A-packets of \cite{abv} are defined as a special class of these
micro-packets.

One thing that is not considered in \cite{abv}, however, is twisted
endoscopy. Our goal is to describe how an outer automorphism may be
introduced into the theory of \cite{abv} in a manner that reflects the
theory of twisted endoscopy introduced by Kottwitz and Shelstad. We
continue by outlining the contents of the paper.

In Section \ref{recall} we present the basic objects of \cite{abv} and
the extended local Langlands Correspondence.  We assume throughout that the
reader has some familiarity with  \cite{abv}.

In Section \ref{auts} we introduce the automorphisms of the basic
objects in \cite{abv}  and provide characterizations of automorphisms
of the  extended groups in terms of invariants. We define a
compatibility condition between an automorphism of an extended group
and an automorphism of its dual. We describe the induced action of the
automorphisms on representations and on $X({^\vee} G^\Gamma )$,
objects appearing in the extended local Langlands Correspondence. 

In Section 4, having obtained actions of automorphisms on objects on
both sides of the extended Local Langlands Correspondence in the
previous section, we formulate and prove a precise equivariance
statement for the correspondence with respect to these actions
(Theorem \ref{equivthm}). 

In the first three subsections of Section 5 we prepare for the
proof of the twisted endoscopic identity.  The definitions of
twisted endoscopic data given in Sections 5.1 and 5.2 are amalgamations
of those given in \cite{ks} and \cite{abv}. Standard endoscopic
lifting, which provides the model for twisted endoscopic lifting, is
reviewed in Section 5.3. As already mentioned, it depends upon a
duality between representations and sheaves, one form of which pairs
irreducible representations with perverse sheaves. Standard endoscopic
lifting is obtained by combining this pairing with the restriction
functor on perverse sheaves.

In Section 5.4 we move on to twisted endoscopic lifting. We define
\emph{twisted} representations and \emph{twisted} perverse sheaves,
and give conditions for a natural pairing between them. This pairing
is then combined with the restriction functor, just as in the standard
case. After this, the only remaining complication is defining the
correct objects to which one applies the lift. A geometric
interpretation of these microlocal objects is given in (25.1)(j)
\cite{abv}. We reformulate them as virtual twisted representations in
(\ref{etadef2}). These technical difficulties aside, the essential endoscopic
lifting identity in its twisted form already appears in \cite{abv} as
Theorem 25.8. One might say without much exaggeration that twisted
endoscopic lifting is proven in that theorem. 

We finish in Section 6 by specializing to the context of twisted
endoscopy for general linear groups as in \cite{arthurbook}. An example for
$\mathrm{GL}_2$ is provided followed by a discussion of the
computational difficulties that arise is higher rank. 

Ultimately, one would like to compare the twisted endoscopic identity
(\ref{endolift2})  of
the final section  with Arthur's twisted endoscopic
identity  (Theorem 2.2.1 
\cite{arthurbook}).  In doing so, one should be able to compare the 
A-packets of Arthur and the A-packets of \cite{abv}  (\emph{cf.} 8
\cite{artreal}).   Recently,
Arancibia, Moeglin and 
Renard have shown that  Arthur's A-packets are identical in some key cases
to those of Adams-Johnson and Barbasch-Vogan (\cite{amr}, \cite{mr}).
Undoubtedly, this indicates that Arthur's A-packets are identical to
those of \cite{abv}.

In closing, we hope that our excursion into \cite{abv} will be useful
to anyone with an interest in this book.  Although our treatment is
not expository, we have made efforts to reveal the structure of
\cite{abv} without overtaxing the reader.

\section{Recollections from Adams-Barbasch-Vogan}
\label{recall}

We briefly recall those objects of
\cite{abv} to which automorphisms may apply.    Everything
recalled here may be found in the first six chapters of \cite{abv}.  

Let $\Gamma$ be the
Galois group $\mathrm{Gal}(\mathbb{C}/\mathbb{R})$.  Throughout $G$ denotes a
connected complex reductive algebraic group.  A \emph{weak extended
  group containing} $G$  is a real Lie group $G^{\Gamma}$ which is an extension
$$1 \rightarrow G \rightarrow G^{\Gamma} \rightarrow \Gamma
\rightarrow 1$$
subject to the condition that every element of $G^{\Gamma}-G$ acts on
$G$ by conjugation as an antiholomorphic automorphism.  Henceforth we
fix a weak extended group $G^{\Gamma}$.  A \emph{strong
  real form} of $G^{\Gamma}$ is an element $\delta \in G^{\Gamma}-G$
such that $\delta^{2}$ is central and has finite order.  The weak
extended group $G^{\Gamma}$ becomes an
\emph{extended group containing} $G$ if it is endowed with 
the $G$-conjugacy class $\mathcal{W}$ of a triple
$(\delta, N,  \chi)$ in which $\delta$ is a strong real
form, $N$ is a 
maximal unipotent subgroup of $G$ normalized by $\delta$, and $\chi$
is a non-degenerate unitary character on the real points of 
$N$.  This triple is 
called a \emph{Whittaker datum}.

Let us denote the inner automorphism of $\delta \in G^{\Gamma}$ by
$\mathrm{Int}(\delta)$.  Every strong real form $\delta$ of
$G^{\Gamma}$ defines an antiholomorphic involution   $\sigma(\delta) =
\mathrm{Int}(\delta)_{|G}$.  The fixed-point subgroup $G(\mathbb{R},
\delta) = G^{\sigma(\delta)}$ of
$\sigma(\delta)$ is a real form of $G$.  The set of such
real forms constitutes an inner class of 
some quasisplit form of $G$ (Proposition 2.14 \cite{abv}).  Two strong
real forms 
are \emph{equivalent} if they are $G$-conjugate.  Equivalent strong real
forms produce isomorphic real forms.  However, it is sometimes also possible for
inequivalent strong real forms to produce isomorphic real forms.

The (weak) extended group $G^{\Gamma}$ may be characterized in terms of
two invariants (Corollary 2.16 and Proposition 3.6 \cite{abv}).    The
first of the two invariants is an automorphism 
$a$ of the canonical based root datum $\Psi_{0}(G)$, which is induced from
conjugation by a(ny) $\delta \in G^{\Gamma}-G$.  To express the second
invariant, we set $Z(G)$ equal to the centre of $G$ and $\sigma_{Z} =
\mathrm{Int}(\delta)_{|Z(G)}$ for a(ny) $\delta \in G^{\Gamma}-G$.
There exists an element $\delta_{q} \in G^{\Gamma} - G$ such that
$G(\mathbb{R}, \delta_{q})$ is a quasisplit real form and
$\delta_{q}^{2} \in Z(G)^{\sigma_{Z}}$.  The second invariant of
$G^{\Gamma}$ is the coset $\bar{z} \in Z(G)^{\sigma_{Z}}/
(1+\sigma_{Z}) Z(G)$ of $\delta_{q}^{2}$ (which is well-defined).  The
pair of invariants $(a, \bar{z})$ determines the weak extended group
$G^{\Gamma}$ up to isomorphism.   When $G^{\Gamma}$ is endowed with a
Whittaker datum $\mathcal{W} = G \cdot (\delta_{0},N,\chi)$, then $z =
\delta_{0}^{2} \in Z(G)^{\sigma_{Z}}$ is a canonical representative for $\bar{z}$.
In this case $z$ is called the second invariant of the extended group
$(G^{\Gamma}, \mathcal{W})$, and the pair $(a,z)$ determines the
extended group up to isomorphism.

One of our principal objects of interest is a \emph{representation of a
strong real form of} $G^{\Gamma}$.  This is a pair $(\pi,\delta)$ in
which $\delta$ is a strong real form of $G^{\Gamma}$ and $\pi$ is an
admissible representation of $G(\mathbb{R}, \delta)$.  Two such pairs
$(\pi, \delta)$ and $(\pi', \delta')$ are \emph{equivalent} if
$g\delta g^{-1} = \delta'$ for some $g \in G$, and $\pi \circ
\mathrm{Int}(g^{-1})$ is infinitesimally equivalent to $\pi'$.  
Let $\Pi(G/\mathbb{R})$ denote the set of (equivalence classes of)
irreducible representations of strong real forms of $G^{\Gamma}$.  

We now recall the objects which are dual  to $G^{\Gamma}$. Let
${^\vee}G$ be the dual group of $G$.  That is to say ${^\vee}G$ is a
connected complex reductive algebraic group whose 
canonical based root datum is dual to that of $G$
$$\Psi_{0}({^\vee}G) = {^\vee} \Psi_{0}(G).$$   A
\emph{weak E-group for} $G$ is an algebraic group
${^\vee}G^{\Gamma}$ which is an extension
$$1 \rightarrow {^\vee}G \rightarrow {^\vee}G^{\Gamma} \rightarrow \Gamma
\rightarrow 1.$$  
Let ${^\vee}G^{\Gamma}$ be a weak E-group for $G$.  
Conjugation by a(ny) element ${^\vee}\delta \in {^\vee}G^{\Gamma}-
{^\vee}G$ induces an automorphism in 
\begin{equation}
\label{autiso}
\mathrm{Aut}\Psi_{0}({^\vee}G)
\cong \mathrm{Aut}\Psi_{0}(G).
\end{equation}  We call this automorphism the
\emph{first invariant} of ${^\vee}G^{\Gamma}$.  The weak E-group
${^\vee}G^{\Gamma}$ for $G$ becomes a \emph{weak E-group for}
$G^{\Gamma}$ if its first invariant equals the first
invariant $a \in \mathrm{Aut}
\Psi_{0}(G)$ of $G^{\Gamma}$ under (\ref{autiso}).  From now on we
assume that ${^\vee}G^{\Gamma}$ is a weak E-group for $G^{\Gamma}$.
This weak E-group becomes an \emph{E-group} for $G$ if we endow it
with a ${^\vee}G$-conjugacy class $\mathcal{D}$ of elements of finite
order in ${^\vee}G^{\Gamma} - {^\vee}G$ whose inner automorphisms
preserve  splittings of ${^\vee}G$ (Proposition 2.11 \cite{abv}).

There is a classification of (weak) E-groups in terms of first and
second invariants, just as there was for extended groups (Proposition
4.4, Proposition 4.7 and Corollary 4.8 \cite{abv}).  The second
invariant of ${^\vee}G^{\Gamma}$ is an element of
$Z({^\vee}G)^{\theta_{Z}}/(1+\theta_{Z}) Z({^\vee}G)$, where
$\theta_{Z} = \mathrm{Int}({^\vee}\delta)_{|Z({^\vee}G)}$ for any ${^\vee}\delta \in
{^\vee}G^{\Gamma} - {^\vee}G$.  The second invariant of an E-group
$({^\vee}G^{\Gamma}, \mathcal{D})$ is the canonical representative of
the previous second invariant
given by the square of any element in $\mathcal{D}$ (Definition 4.6
\cite{abv}). 

The dual objects to representations of strong real forms are
\emph{complete geometric parameters}.  To describe these
parameters we must first introduce a complex variety $X({^\vee}G^{\Gamma})$,
which reparametrizes the set of L-parameters pioneered by Langlands.  
Towards this end, let ${^\vee}\mathfrak{g}$ be the Lie algebra of
${^\vee}G$.  Let 
 $\lambda \in {^\vee}\mathfrak{g}$ be a semisimple element and set
$\mathcal{O} = \mathrm{Ad}({^\vee}G) \cdot \lambda \subset
{^\vee}\mathfrak{g}$.  The subalgebra $\mathfrak{n}(\lambda) \subset
{^\vee}\mathfrak{g}$ consists of the positive integral eigenspaces
of $\mathrm{ad}(\lambda)$.  Let $N(\lambda) =
\exp(\mathfrak{n}(\lambda))$, $L(\lambda)$ be the centralizer in
${^\vee}G$ of $\lambda$, and $P(\lambda) = L(\lambda)N(\lambda)$.  The
\emph{canonical flat through} $\lambda$ is the set
$\mathcal{F}(\lambda) =
\mathrm{Ad}(P(\lambda))\cdot \lambda$.  The set
$\mathcal{F}(\mathcal{O})$ of canonical flats  which are conjugate to
$\mathcal{F}(\lambda)$ encodes information about L-parameters of
representations with infinitesimal character $\lambda$.  This
corresponds to the restriction of these L-parameters to
$\mathbb{C}^{\times}$ in the Weil group $W_{\mathbb{R}} =
\mathbb{C}^{\times} \coprod j\mathbb{C}^{\times}$ (Proposition 5.6
\cite{abv}). The set which encodes the values of these L-parameters at
$j$ is 
$$\mathcal{I}(\mathcal{O}) = \{y \in {^\vee}G^{\Gamma} -
     {^\vee}G : y^{2} \in  \exp(2 \pi i \mathcal{O}) \}.$$
The set of \emph{geometric parameters} (for $\mathcal{O}$) is
$$X(\mathcal{O}, {^\vee}G^{\Gamma})=  \{ (y, \mathcal{F}(\lambda')):
\lambda' \in \mathcal{O},\ y \in \ {^\vee}G^{\Gamma} -
     {^\vee}G,  \ y^{2} = \exp(2
\pi i \lambda')\}.$$
This set is a fibre product of $\mathcal{F}(\mathcal{O})$ and
$\mathcal{I}(\mathcal{O})$ which carries a natural structure of a
complex algebraic variety (Proposition 6.16 \cite{abv}).  By
definition, the set of all geometric parameters is the disjoint
union
$$X({^\vee}G^{\Gamma}) = \coprod_{\mathcal{O}} X(\mathcal{O}, {^\vee}G^{\Gamma}).$$

The dual
group ${^\vee}G$ acts on geometric parameters by conjugation and this
action defines the notion of equivalence for geometric parameters.
According to Proposition 6.17 \cite{abv}, the
set of equivalence classes of geometric parameters is in bijection with
the equivalence classes of the aforementioned L-parameters 
(ignoring the concept of relevance).

The local Langlands Correspondence, as originally conceived, is a
bijection between (equivalence classes of) L-parameters and L-packets.
Adams, Barbasch and Vogan  refine and extend
the local Langlands Correspondence to a bijection between (equivalence
classes of) complete geometric parameters and (equivalence
classes of) representations of strong real forms.  They do this by
supplementing each equivalence class of a geometric parameter
with the representation of a finite group.  To be  precise, let $x
= (y, \Lambda) \in X(\mathcal{O}, {^\vee}G^{\Gamma})$  be a
geometric parameter and $S = {^\vee}G \cdot x$ be its
equivalence class.  Let ${^\vee}G_{x}$ be the isotropy group of $x$
and ${^\vee}G^{alg}_{x}$ be the preimage of this isotropy group in
the universal algebraic cover 
\begin{equation}
\label{algcover}
1 \rightarrow \pi_{1}({^\vee}G)^{alg} \rightarrow {^\vee}G^{alg}
\rightarrow {^\vee}G \rightarrow 1.
\end{equation}
The finite group in question is the component group 
$$  A_{x}^{loc, alg}= 
{^\vee}G^{alg}_{x}/ ({^\vee}G^{alg}_{x})_{0}$$
(Definition 7.6
\cite{abv}).  
Experts
will recognize that this group is an enlargement of the usual Langlands
component group of an L-parameter.  It is therefore fitting to call
this enlargement the \emph{Langlands component group for} $x$ (or
$S$).  It is important to realize that $A_{x}^{loc, alg}$ is abelian
(page 61 \cite{abv}).
It therefore makes sense to identify $A_{S}^{loc,alg}$ with 
$A_{x}^{loc, alg}$ and speak of a representation of $A_{S}^{loc,alg}$.

The geometric parameter $x = (y,\Lambda)$ becomes a \emph{(local)
  complete geometric parameter} if it is paired with an irreducible
representation  $\tau$ of $A_{S}^{loc, alg}$.  As $A_{S}^{loc, alg}$
is  abelian, the representation $\tau$ is
actually a character.  A \emph{complete geometric parameter} is a pair
of the form 
$(S, \tau \circ \mathrm{Int}({^\vee}G^{alg}))$, in other words an
equivalence class of a local complete geometric parameter.  The set of
complete geometric parameters is denoted by $\Xi({^\vee}G^{\Gamma})$.

The  extension of the original Langlands Correspondence for real
groups takes the shape of a bijection
\begin{equation}
\label{llc1}
\Pi(G/\mathbb{R}) \stackrel{LLC}{\longleftrightarrow}
\Xi({^\vee}G^{\Gamma})
\end{equation}
in Theorem 1.18 \cite{abv}.  It is an extension in the sense that it
combines the traditional Langlands correspondence for all inner forms of a
quasisplit  form  into a single bijection.

There are two important variants of (\ref{llc1}) in \cite{abv} which we sketch
briefly.  Correspondence (\ref{llc1}) is only
valid for an E-group $^{\vee}G^{\Gamma}$ whose second invariant $z \in
Z({^\vee}G)^{\theta_{Z}}$
is trivial.  When $z$ is non-trivial,
rather than considering representations of a strong real form
$G(\mathbb{R}, \delta)$, one must consider representations of a
certain profinite cover
\begin{equation}
\label{cancov}
1 \rightarrow \pi_{1}(G)^{can} \rightarrow
G(\mathbb{R},\delta)^{can} \rightarrow G(\mathbb{R},\delta)
\rightarrow 1
\end{equation}
(Definition 10.3 \cite{abv}) whose restrictions to $\pi_{1}(G)^{can}$ are related
to $z$.  The set of (equivalence classes of) these \emph{canonical projective
  representations of type} $z$ is denoted by $\Pi^{z}(G/\mathbb{R})$.
For non-trivial second invariant $z$, 
correspondence (\ref{llc1}) is expressed 
as a bijection
\begin{equation}
\label{llc2}
\Pi^{z}(G/\mathbb{R}) \stackrel{LLC}{\longleftrightarrow}
\Xi^{z}({^\vee}G^{\Gamma}) = \Xi({^\vee}G^{\Gamma})
\end{equation}
(Theorem 10.4 \cite{abv}).

One may consider a quotient $Q$ of the algebraic fundamental
group $\pi({^\vee}G)^{alg}$ in (\ref{algcover}) and use the smaller
cover
$$1 \rightarrow Q \rightarrow {^\vee}G^{Q} \rightarrow {^\vee}G
\rightarrow 1$$
in place of $^{\vee}G^{alg}$ to define a \emph{complete geometric
  parameter of type} $Q$ (Definition 7.6 \cite{abv}).  The difference
here is that the representation in the complete geometric
parameter is a representation of a group
$$A_{S}^{loc, Q} = A_{x}^{loc, Q} = 
{^\vee}G^{Q}_{x}/ ({^\vee}G^{Q}_{x})_{0}$$
The set of
(equivalence classes of) these parameters is denoted by
$\Xi^{z}({^\vee}G^{\Gamma})^{Q}$.  

Looking at (\ref{llc2}) one expects that
$\Xi^{z}({^\vee}G^{\Gamma})^{Q}$ corresponds bijectively to 
a subset of $\Pi^{z}(G/\mathbb{R})$. To say what this subset is
requires more details about $Q$. The quotient $Q$ is obtained from a
closed subgroup $K_{Q}$ 
\begin{equation}
\label{qq}
1 \rightarrow K_{Q} \rightarrow \pi_{1}({^\vee}G)^{alg} \rightarrow Q
\rightarrow 1.
\end{equation}
Let $\hat{Q}$
be the character group of $Q$. One may view $\hat{Q}$ as the subgroup
of characters of $\pi_{1}({^\vee}G)^{alg}$ which are trivial
on $K_{Q}$.  By Lemma 10.9 (a) \cite{abv}, there is an isomorphism
between the characters of $\pi_{1}({^\vee}G)^{alg}$ and the
finite-order elements in $Z(G)$.  Using this isomorphism, the character
group $\hat{Q}$ may be identified with a finite subgroup $J$ of
$Z(G)$.  Define  $\Pi^{z}(G,\mathbb{R})_{J} =
\Pi^{z}(G/\mathbb{R})_{\hat{Q}}$ to consist of those $(\pi,\delta) \in
\Pi^{z}(G/\mathbb{R})$ satisfying $\delta^{2} \in z J$ ($z$ is the
second invariant of $(G,\mathcal{W})$).  Theorem 10.11 \cite{abv}
broadens (\ref{llc2}) to include bijections of the shape
\begin{equation}
\label{llc3}
\Pi^{z}(G/\mathbb{R})_{\hat{Q}} \stackrel{LLC}{\longleftrightarrow}
\Xi^{z}({^\vee}G^{\Gamma})^{Q}.
\end{equation}

\section{Automorphisms}
\label{auts}

The purpose of this section is to show how automorphisms of  real
reductive groups may be introduced into the
framework of \cite{abv}.  We continue with the notation and
assumptions of Section \ref{recall}.
\begin{definition}[]
\label{wextaut}
An \emph{automorphism of a weak extended group} $G^{\Gamma}$ is a group 
automorphism $\vartheta^{\Gamma}$
of $G^{\Gamma}$ whose restriction $\vartheta = \vartheta^{\Gamma}_{|G}$  is
a holomorphic (algebraic) automorphism of $G$.    Such an automorphism
is an \emph{automorphism of an extended group} $(G^{\Gamma}, \mathcal{W})$
if the Whittaker datum $\mathcal{W} = G \cdot (\delta_{0}, N, \chi)$ is
equal to $G \cdot (\vartheta^{\Gamma}(\delta_{0}), \vartheta(N), \chi
\circ \vartheta^{-1})$.  
\end{definition}

\subsection{Characterizations of automorphisms}

The next two propositions characterize the relationship between automorphisms of
$G$ and automorphisms of $G^{\Gamma}$.

\begin{prop}
Suppose $\vartheta^{\Gamma}$ is an automorphism of a weak extended group
$G^{\Gamma}$ with invariants $(a,\bar{z})$, 
and let $\vartheta = \vartheta^{\Gamma}_{|G}$.  Then
\begin{enumerate}
\item The automorphism $\vartheta$ passes to an automorphism
  $\Psi_{0}(\vartheta)$ of  $\Psi_{0}(G)$ which commutes with $a$.


\item The automorphism $\vartheta$ passes to an automorphism 
$\bar{\vartheta}$ of $Z(G)^{\sigma_{Z}}/
(1+\sigma_{Z}) Z(G)$ satisfying $\bar{\vartheta}(\bar{z}) = \bar{z}$.

\end{enumerate}
\label{thetares}
\end{prop}
\begin{proof}
The construction of the automorphism $\Psi_{0}(\vartheta)$ is
well-known (see Proposition 2.11 \cite{abv}).  To see that it commutes
with $a$ we note that $a = \Psi_{0}(\sigma(\delta)) =
\Psi_{0}(\mathrm{Int}(\delta)_{|G})$ for a(ny)  $\delta \in G^{\Gamma} -
G$ (Proposition 2.12 and Proposition 2.16 \cite{abv}).  It is easily
computed that 
$$\mathrm{Int}(\vartheta^{\Gamma}(\delta))_{|G}  = 
\vartheta \circ \mathrm{Int}(\delta)_{|G} \circ \vartheta^{-1}.$$
Since the elements $\vartheta^{\Gamma}(\delta)$ and $\delta$ both belong to 
$G^{\Gamma}-G$ we have
$$ a =
\Psi_{0}(\mathrm{Int}(\delta)_{|G})=
\Psi_{0}(\mathrm{Int}(\vartheta^{\Gamma}(\delta))_{|G}) = \Psi_{0}(\vartheta
\circ \mathrm{Int}(\delta)_{|G} \circ 
\vartheta^{-1}).$$
It follows from the definitions that the expression on the right is
equal to 
$$\Psi_{0}(\vartheta) \circ \Psi_{0}(\mathrm{Int}(\delta)_{|G}) \circ
\Psi_{0}(\vartheta)^{-1}  = \Psi_{0}(\vartheta) \circ
a \circ 
\Psi_{0}(\vartheta)^{-1}$$
(\emph{cf.} page 34 \cite{abv}).  This proves the first assertion.
Similar arguments allow us to deduce  that 
$$\sigma_{Z} = \mathrm{Int}(\delta)_{|Z(G)} = \vartheta \circ
\mathrm{Int}(\delta) \circ\vartheta^{-1}_{|Z(G)}=  \vartheta \circ
\sigma_{Z} \circ \vartheta^{-1}_{|Z(G)}.$$
This equation justifies the existence of the
automorphism $\bar{\vartheta}$.
For the final assertion, recall from Section \ref{recall}
that $\bar{z}$ is
equal to the coset of $\delta_{q}^{2}$ in $Z(G)^{\sigma_{Z}}/
(1+\sigma_{Z}) Z(G)$, where 
$\sigma(\delta_{q})$ is a quasisplit real form of $G$ in the inner class
defined by $a$.  Let $B$ be a Borel subgroup preserved by $\sigma(\delta_{q})$
and set $\delta_{q}' = \vartheta^{\Gamma}(\delta_{q})$.  Then
$$\vartheta(B) = \vartheta \circ \mathrm{Int}(\delta_{q}) \circ
\vartheta^{-1} (\vartheta(B)) = \mathrm{Int}(\vartheta^{\Gamma}(\delta_{q}))
(\vartheta(B)) = \sigma(\delta_{q}') (\vartheta(B))$$
implies that $\sigma(\delta_{q}')$ is also a quasisplit form.  Therefore,
according to
Corollary 2.16 (a) \cite{abv}, 
the element $\bar{z}$ is also equal to the coset of 
 $(\delta_{q}')^{2} = \vartheta(\delta_{q}^{2})$ in $Z(G)^{\sigma_{Z}}/
(1+\sigma_{Z}) Z(G)$.  This is equivalent to $\bar{\vartheta}(\bar{z}) = \bar{z}$.
\end{proof}

\begin{prop}
\label{thetares1}
Suppose $\vartheta$ is a holomorphic automorphism of $G$, and $G^{\Gamma}$ 
is a weak extended group for $G$ with invariants $(a, \bar{z})$.  Suppose
further that $\delta_{q} \in G^{\Gamma} - G$ yields a quasisplit form. 
If $\Psi_{0}(\vartheta)$ commutes with $a$
 then $\vartheta$ passes to an automorphism $\bar{\vartheta}$ of
$Z(G)^{\sigma_{Z}}/
(1+\sigma_{Z}) Z(G)$.  If in addition $\bar{\vartheta}(\bar{z}) = \bar{z}$, then
 $\vartheta$ extends to an automorphism $\vartheta^{\Gamma}$ of $G^{\Gamma}$.
Any two such extensions differ at their value on $\delta_{q}$ 
by a multiple of an element in $\ker(1+ \sigma_{Z})$.
\end{prop}
\begin{proof}
Suppose  $\Psi_{0}(\vartheta)$ commutes with $a$. Then the existence of
$\bar{\vartheta}$ follows as in the proof of the previous
proposition.  The commutativity of $\Psi_{0}(\vartheta)$ and
$a$ also implies
\begin{equation}
\label{acommute}
\Psi_{0}(\sigma(\delta_{q})) = \Psi_{0}(\vartheta \circ
\sigma(\delta_{q}) \circ \vartheta^{-1}).
\end{equation}
Suppose $\bar{\vartheta}(\bar{z}) = \bar{z}$.  By (\ref{acommute})
and Proposition 2.12 \cite{abv},
the form $\sigma_{q} = \sigma(\delta_{q})$ is
equivalent to the form $\vartheta \sigma_{q} \vartheta^{-1}$.  It 
follows from Proposition 2.14 \cite{abv} that there exists $\delta_{1}
\in G^{\Gamma}-G$, $g \in G$ and $z_{1} \in 
Z(G)$ such that $\sigma(\delta_{1}) =\vartheta \circ \sigma_{q} 
\circ \vartheta^{-1}$ and 
$\delta_{1} = z_{1} g \delta_{q} g^{-1}$.  By replacing $\delta_{1}$
with $z_{1}^{-1} \delta_{1}$ we may assume without loss of generality
that $z_{1}$ is trivial and $\delta_{1} = g\delta_{q} g^{-1}$. 
Let $z = \delta_{q}^{2} = \delta_{1}^{2}$ be a coset representative of $\bar{z}$
(Corollary 2.16 \cite{abv}).  
Then $\bar{\vartheta}(\bar{z}) = \bar{z}$ amounts to
$\vartheta(z) = z_{2} \sigma_{Z}(z_{2}) \, z$ for some $z_{2} \in
Z(G)$. Set $\vartheta^{\Gamma}(\delta_{q}) = z_{2} \delta_{1}$ so that
$\vartheta^{\Gamma}(\delta_{q})^{2} = \vartheta(z)$ and $\vartheta
\circ \sigma_{q} \circ
 \vartheta^{-1} = \sigma(\vartheta^{\Gamma}(\delta_{q}))$.  To prove that
$\vartheta^{\Gamma}$ defines an extension of  $\vartheta$ to an automorphism 
of $G^{\Gamma}$ we use the equations of (2.17)(b) \cite{abv}.  We compute
\begin{align*}
\vartheta^{\Gamma}(g_{1}\delta_{q} g_{2} \delta_{q}) & = \ \vartheta (
g_{1} \sigma_{q}(g_{2}) z) \\
& = \vartheta(g_{1}) \ \vartheta \circ \sigma_{q} \circ \vartheta^{-1}(\vartheta
(g_{2})) \ \vartheta^{\Gamma}(\delta_{q})^{2}\\
& = \vartheta^{\Gamma}(g_{1} \delta_{q}) \ \vartheta^{\Gamma}(g_{2} \delta_{q})
\end{align*}
and 
\begin{align*}
\vartheta^{\Gamma}(g_{1} \delta_{q} g_{2}) & = \vartheta^{\Gamma}(g_{1} \sigma_{q}
(g_{2}) \delta_{q})\\
& = \vartheta(g_{1}) \ \vartheta \circ \sigma_{q} \circ \vartheta^{-1}(\vartheta
(g_{2})) \ \vartheta^{\Gamma}(\delta_{q})\\
& = \vartheta^{\Gamma}(g_{1}\delta_{q}) \ \vartheta^{\Gamma}(g_{2}).
\end{align*}
For the final assertion, observe that if $\vartheta^{\Gamma}(\delta_{q}) = 
x z_{2} \delta_{1}$ for some $x \in G$ then $x$ must belong to $Z(G)$
for $\vartheta
\circ \sigma_{q} \circ
 \vartheta^{-1} = \sigma(\vartheta^{\Gamma}(\delta_{q}))$ to remain
 true. Furthermore  $x\sigma_{Z}(x) = 1$ 
for the required identity  $\vartheta^{\Gamma}(\delta_{q})^{2} = \vartheta(z)$.
\end{proof}

There is a parallel description of automorphisms for E-groups.
\begin{definition}
Suppose $^{\vee}{G}^{\Gamma}$ is a weak E-group for $G$. 
 An \emph{automorphism} $\varsigma$ of $^{\vee}G^{\Gamma}$ is simply
 a holomorphic (algebraic) group automorphism.  Such an automorphism
 is an \emph{automorphism of an E-group} 
$(^{\vee}G^{\Gamma}, \mathcal{D})$ for $G$ if the conjugacy class
 $\mathcal{D}$ is preserved by $\varsigma$.
\label{sigmares}
\end{definition}

Since any algebraic automorphism of ${^\vee}G^{\Gamma}$ preserves the
identity component ${^\vee}G$ this definition is analogous to
Definition \ref{wextaut}.   There are also analogues of Propositions
\ref{thetares} and \ref{thetares1} with simpler proofs.   The element
$\delta_{q} \in G^{\Gamma}-G$ in the proof of Proposition
\ref{thetares1} needs only to be  replaced by an 
element ${^\vee}\delta \in {^\vee}G^{\Gamma} - {^\vee}G$ whose inner automorphism is 
\emph{distinguished} (\emph{i.e.} preserves a splitting).    One 
may then use the 
fact that distinguished automorphisms with respect to different splittings
are $\mathrm{Int}({^\vee}G)$-conjugate, just as quasisplit forms in an 
inner class are
$\mathrm{Int}(G)$-conjugate.  The remaining details are left to the
interested reader.

\begin{prop}
Suppose $\varsigma$ is an automorphism of the weak E-group ${^\vee}G^{\Gamma}$
for $G$ with invariants $(a,\bar{z})$. 
Then 
\begin{equation}
\label{abequ}
a = \Psi_{0}(\mathrm{Int}({^\vee}\delta)_{|{^\vee}G}) = 
\Psi_{0}(\varsigma \circ \mathrm{Int}({^\vee}\delta)
\circ \varsigma^{-1}_{|{^\vee}G}) = \Psi_{0}(\varsigma_{|{^\vee}G}) \circ a \circ
\Psi_{0} (\varsigma_{|{^\vee}G})^{-1}
\end{equation}
for any ${^\vee} \delta \in {^\vee}G^{\Gamma} - ^{\vee}G$, and $\varsigma$
passes to an automorphism of $\bar{\varsigma}$ of 
$Z({^\vee}G)^{\theta_{Z}}/ (1+\theta_{Z}) Z({^\vee}G)$ satisfying 
$\bar{\varsigma}(\bar{z}) = \bar{z}$.
Conversely, any automorphism of ${^\vee}G$, satisfying the above two equations
for an element ${^\vee}\delta$ with distinguished inner automorphism, extends to
an automorphism of ${^\vee}G^{\Gamma}$.  Moreover any two such extensions
differ by a multiple of an element in $\ker(1+\theta_{Z})$.
\label{sigmares1}
\end{prop}
\begin{exmp}
The choice of extensions in Proposition \ref{sigmares1} is
parameterized by $\ker(1+\theta_{Z})$, and this kernel is related to
twisting by one-cocycles (2.1 \cite{ks}).  Suppose $\varsigma$
is an automorphism ${^\vee}G^{\Gamma}$.  By Theorem 6.2.2 \cite{weibel},
$$H^{1}(\Gamma, Z({^\vee}G)) \cong
\ker(1+\theta_{Z})/ (1-\theta_{Z}) Z({^\vee}G)$$
where $\Gamma$ acts on $Z({^\vee}G)$ by $\theta_{Z}$.  Let  
$\mathbf{a} \in H^{1}(\Gamma, Z({^\vee}
G))$ and  $\mathsf{a}$ be a  one-cocycle in the cohomology class 
$\mathbf{a}$.
By the above isomorphism, $\mathsf{a}$ may be
identified with $\mathsf{a}(\gamma) \in \ker(1+\theta_{Z})$, where
$\gamma \in \Gamma$  is the non-trivial element.
One may define another automorphism $\varsigma_{\mathsf{a}}$ of
${^\vee}G^{\Gamma}$ by following Proposition \ref{sigmares1}
(\emph{cf.} page 17 \cite{ks}).  More explicitly, let ${^\vee}\delta
\in {^\vee}G^{\Gamma} - {^\vee}G$ 
be an element such that $\mathrm{Int}({^\vee}\delta)$ fixes a
splitting.  Then define
$$\varsigma_{\mathsf{a}}(g\, {^\vee}\delta) = \mathsf{a}
(\gamma)^{-1}  \varsigma(g \, {^\vee}\delta) \mbox{ and }
\varsigma_{\mathsf{a}}(g) =  \varsigma(g),
 \ g \in
{^\vee}G.$$
This definition is valuable even when $\varsigma$ is trivial (see
Example \ref{24ex}).  
\label{cocycletwist}
\end{exmp}

We are assuming that ${^\vee}G^{\Gamma}$ is a weak E-group for $G$.
This means that the groups $G^{\Gamma}$ and ${^\vee}G^{\Gamma}$ have a common
 first invariant $a \in \mathrm{Aut}(\Psi_{0}(G)) \cong
 \mathrm{Aut}(\Psi_{0}({^\vee}G))$. 
It is natural to anticipate some sort of compatibility between an
automorphism of $G^{\Gamma}$ and an automorphism of ${^\vee}G^{\Gamma}$
stemming from the isomorphism $\mathrm{Aut}(\Psi_{0}(G)) \cong
 \mathrm{Aut}(\Psi_{0}({^\vee}G))$.  Before making this explicit,
 let us consider an example motivated by 1.2 \cite{arthurbook}.
\begin{exmp}
\label{autex}
Let $G = \mathrm{GL}_{N}$ together with the automorphism $\vartheta$ equal to
inverse-transpose composed with
$$\mathrm{Int}\left( \left[ \begin{smallmatrix}
0 & \cdots & 0 & 1\\
\vdots &  &\iddots
 &0\\
0 & 1& & \vdots\\
1 & 0& \cdots& 0 
\end{smallmatrix}\right] \right).$$  
We define  a weak extended group $G^{\Gamma}$ by taking $\delta_{q}
\in G^{\Gamma} - G$ 
to act on $G$ by complex conjugation.  This corresponds to the inner class of
the split form of $G$.  The
first invariant $a = \Psi_{0}(\sigma(\delta_{q}))$ and
second invariant $\bar{z} = \overline{\delta_{q}^{2}}$ are both trivial.  By
Proposition \ref{thetares1} (and its proof), $\vartheta$ extends to an 
automorphism $\vartheta^{\Gamma}$ fixing $\delta_{q}$.  Let $\mathcal{W} =
G \cdot (\delta_{q}, N, \chi)$, where $N$ is the upper-triangular
unipotent subgroup and 
$$\chi\left( \left[
\begin{smallmatrix}
1 & a_{1} & * & *\\
0 & 1 & a_{2} & *\\
\vdots & \dots & \ddots & *\\
0 & \cdots & 0 & 1
\end{smallmatrix}
\right]\right) = \prod_{j=1}^{n-1} \chi'(a_{j})$$
for some fixed non-trivial character $\chi'$ of $\mathbb{R}$.
One may verify that $(\vartheta^{\Gamma}(\delta_{q}), \vartheta(N), \chi \circ
\vartheta^{-1})$ is conjugate to $(\delta_{q}, N, \chi)$ under the
diagonal matrix 
$\mathrm{diag}(1,-1,1,-1,\ldots)$.  Consequently, $\vartheta^{\Gamma}$ is 
an automorphism of the extended group $(G^{\Gamma}, \mathcal{W})$.  A
slightly cleaner arrangement of essentially the same example is to
define $\vartheta$ as inverse-transpose composed with the inner
automorphism of
$$\tilde{J} = \mathrm{diag}(1,-1,1,-1,\ldots) \  \mathrm{Int}\left(
\left[ \begin{smallmatrix} 
0 & \cdots & 0 & 1\\
\vdots &  &\iddots
 &0\\
0 & 1& & \vdots\\
1 & 0& \cdots& 0 
  \end{smallmatrix}\right] \right),$$
in which case $(\vartheta^{\Gamma}(\delta_{q}), \vartheta(), \chi \circ
\vartheta^{-1})$ is actually equal to $(\delta_{q}, N, \chi)$ (see
(1.2.1) \cite{arthurbook}).  Let us continue with $\vartheta$ defined
in this way.  

The automorphism $\Psi_{0}(\vartheta)$ of the canonical based root
datum for $G$ transfers to 
an automorphism of the canonical based root datum for ${^\vee}G$ and 
the latter corresponds to an $\mathrm{Int}({^\vee}G)$-conjugacy class of
algebraic automorphisms of ${^\vee}G$ 
(Proposition 2.11 \cite{abv}).  In this
example $G = {^\vee}G = \mathrm{GL}_{N}$, and  $^{\vee}\vartheta = \vartheta$ is
an automorphism in this conjugacy class.  

We define  $^{\vee}G^{\Gamma} = {^\vee}G \times \Gamma$ to be
Langlands' L-group.  In this case the first and second 
invariants of ${^\vee}G^{\Gamma}$ are again trivial. 
By Proposition \ref{sigmares1}, there is an automorphism
${^\vee}\vartheta^{\Gamma}$ of ${^\vee}G^{\Gamma}$ which is trivial on $\Gamma$
and extends ${^\vee}\vartheta$.  Let
$\mathcal{D}$ be the singleton containing the non-trivial element of
$\Gamma$ in the direct product.  Then ${^\vee}\vartheta^{\Gamma}$ is an
automorphism of the E-group $({^\vee}G^{\Gamma}, \mathcal{D})$ for $G$.  
\end{exmp}

In this example the $\mathrm{Int}({^\vee}G)$-conjugacy class of
automorphisms  depends on a choice of splitting made in Proposition
2.11 \cite{abv}, and each automorphism in this class is
\emph{distinguished} in the sense that it preserves some splitting of
${^\vee}G$ (page 34 \cite{abv}).\footnote{This suggests that perhaps
  the dual object to $\vartheta$ ought to be 
the $\mathrm{Int}({^\vee}G)$-conjugacy class, and not a particular
representative.}  We say that an automorphism of ${^\vee}G^{\Gamma}$
is \emph{distinguished} if its restriction to ${^\vee}G$ is distinguished.
 Furthermore, any
automorphism in the above conjugacy class is compatible with 
$\vartheta^{\Gamma}$ in the following sense.
\begin{definition}
  \label{compdef}
Suppose $\vartheta^{\Gamma}$ is an automorphism of $G^{\Gamma}$ and
${^\vee}\vartheta^{\Gamma}$ is an automorphism of ${^\vee}G^{\Gamma}$,
a weak E-group for $G$.  Then $\vartheta^{\Gamma}$ and
${^\vee}\vartheta^{\Gamma}$ are \emph{compatible} if the dual action of
$\Psi_{0}(\vartheta^{\Gamma}_{|G})$ on the root datum
$\Psi_{0}({^\vee}G)$ is equal to 
that of $\Psi_{0}({^\vee}\vartheta_{|G}^{\Gamma})$.
\end{definition}

Let us now reconsider the constructions of Example \ref{autex} more
generally.  Let $\vartheta^{\Gamma}$ be an automorphism of a weak
extended group $G^{\Gamma}$ and $\vartheta =
\vartheta^{\Gamma}_{|G}$.  Then there is an $\mathrm{Int}
({^\vee}G)$-conjugacy class of automorphisms of ${^\vee}G$ which correspond
to $\vartheta$.   Any automorphism in this conjugacy class commutes
with the first invariant of ${^\vee}G^{\Gamma}$ as in equation
(\ref{abequ}).  However, Proposition \ref{sigmares1} tells us that
such an automorphism extends to automorphism of ${^\vee}G^{\Gamma}$ if
and only if it also preserves the second invariant $z$.  Suppose that
one of the automorphisms in the conjugacy class preserves $z$.  In
this case  \emph{any} automorphism 
in the conjugacy class extends to an automorphism 
 of ${^\vee}G^{\Gamma}$ by Proposition \ref{sigmares1}.
Indeed, composition by inner automorphisms has no effect on
the canonical based root datum, and has no effect on evaluation at the central
element $z$.  The distinguished automorphisms of ${^\vee}G^{\Gamma}$ which are
compatible with $\vartheta^{\Gamma}$ are all obtained in
this manner.

The order of distinguished automorphisms shall be important for
endoscopy.  Any distinguished automorphism of ${^\vee}G$ is of finite
order modulo the centre (16.5
\cite{humphreys}).  As the following example shows, a distinguished
automorphism of  ${^\vee}G^{\Gamma}$ need not be of finite order even
if its restriction to ${^\vee}G$ is.
\begin{exmp}
  Let ${^\vee}G = \mathrm{GL}_{1}$, and
  ${^\vee}G^{\Gamma} = {^\vee}G \rtimes \Gamma$ with non-trivial
  $\gamma \in \Gamma$ acting on ${^\vee}G$ by inversion.  In this case
  all automorphisms are vacuously distinguished.  Moreover,  $\ker(1
  + \theta_{Z}) = Z({^\vee}G) = {^\vee}G$, and using  Proposition
  \ref{sigmares1} we may extend the trivial automorphism on ${^\vee}G$
  to an automorphism ${^\vee}\vartheta^{\Gamma}$ on
  ${^\vee}G^{\Gamma}$ by choosing $z_{1} \in {^\vee}G$ and defining
  ${^\vee}\vartheta^{\Gamma}({^\vee}\delta) = z_{1} {^\vee}\delta$ for
  any ${^\vee}\delta \in {^\vee}G^{\Gamma} - {^\vee}G$.  It is simple
  to verify that ${^\vee}\vartheta^{\Gamma}$ is of finite order if and
  only if $z_{1}$ is of finite order.
\end{exmp}
\begin{prop}
  \label{finord}
  Suppose ${^\vee}G$ is semisimple and ${^\vee}\vartheta^{\Gamma}$ is
  a distinguished automorphism 
  of a weak E-group ${^\vee}G^{\Gamma}$ whose restriction
  ${^\vee}\vartheta$ to ${^\vee}G$ is of order $m$.  Then
  ${^\vee}\vartheta^{\Gamma} ({^\vee}\delta) = z_{1} {^\vee}\delta$ for some
  ${^\vee}\delta \in {^\vee}G^{\Gamma} - {^\vee}G$ and $z_{1} \in
  Z({^\vee}G)$.  Furthermore 
  ${^\vee}\vartheta^{\Gamma}$ is of finite order if and only if 
$z_{2} = (1 + {^\vee}\vartheta + \cdots +
  {^\vee}\vartheta^{m-1})(z_{1})$ is of finite order.
\end{prop}
\begin{proof}
There exists an element ${^\vee}\delta \in {^\vee}G^{\Gamma} -
{^\vee}G$ which preserves the same splitting ${^\vee}\vartheta$ does
(Proposition 2.8 \cite{abv}).  It follows in turn that
${^\vee}\vartheta^{\Gamma}( {^\vee}\delta)$ preserves this splitting
and that ${^\vee}\vartheta^{\Gamma}( {^\vee}\delta) = z_{1}
{^\vee}\delta$ for some $z_{1} \in Z({^\vee}G)$.

It is easily seen that
$({^\vee}\vartheta^{\Gamma})^{\ell}(g \,{^\vee}\delta) = g\,
{^\vee}\delta$ for all $g \in {^\vee}G$ only if $m$ divides $\ell$.
The second assertion now follows from
$({^\vee}\vartheta^{\Gamma})^{mk}(g \,{^\vee}\delta) = gz_{2}^{k}\,
{^\vee}\delta$. 
\end{proof}

\subsection{Actions on $\Pi(G/\mathbb{R})$ and $\Xi({^\vee}G^{\Gamma})$}
\label{actions1}

Suppose $(G^{\Gamma}, \mathcal{W})$ is an extended group for $G$,
$\vartheta^{\Gamma}$ is an automorphism thereof and $\vartheta =
\vartheta^{\Gamma}_{|G}$. 
Suppose further that
$({^\vee}G^{\Gamma}, \mathcal{D})$ is an E-group for $G$ with second
invariant  $z \in Z({^\vee}G)^{\theta_{Z}}$, and that one (and hence
all) of the dual 
automorphisms ${^\vee}\vartheta$ in the  conjugacy class  discussed 
at the end of the previous section preserves the${^\vee}G$-conjugacy class
$\mathcal{D} \subset {^\vee}G^{\Gamma}$.
Then ${^\vee}\vartheta$ fixes $z$, and by Proposition  \ref{sigmares1} 
we may extend ${^\vee}\vartheta$ to an automorphism 
$^{\vee}\vartheta^{\Gamma}$ of $({^\vee}G^{\Gamma},\mathcal{D})$.  By
construction, the automorphisms $\vartheta^{\Gamma}$ and
${^\vee}\vartheta^{\Gamma}$ are compatible. 

The automorphism $\vartheta^{\Gamma}$ has a straightforward action upon a 
representation of a strong real form $(\pi, \delta)$, namely
\begin{equation}
\label{repact}
\vartheta^{\Gamma} \cdot (\pi, \delta) = (\pi \circ \vartheta, 
(\vartheta^{\Gamma})^{-1}(\delta)).
\end{equation}
This action passes to the set of equivalence classes $\Pi(G/\mathbb{R})$.  
We say that $(\pi, \delta)$ is $\vartheta^{\Gamma}$\emph{-stable} if it is
equivalent to $\vartheta^{\Gamma} \cdot (\pi, \delta)$ (Definition 2.13
\cite{abv}). 

 There is also a straightforward
action of the automorphism ${^\vee}\vartheta^{\Gamma}$ 
on a complete geometric parameter.  
To define this action clearly, we fix a representative $((y, \Lambda),\tau)$
for a complete geometric parameter (Section \ref{recall}), where  
$(y, \Lambda) \in  X({^\vee}G^{\Gamma})$.  
We define ${^\vee}\vartheta^{\Gamma} \cdot (y, \Lambda) = 
({^\vee}\vartheta^{\Gamma}(y), d\,{^\vee}\vartheta(\Lambda))$, where
$d\,{^\vee}\vartheta$ 
is the differential of ${^\vee}\vartheta$. It is easily verified that
${^\vee}\vartheta^{\Gamma} \cdot (y, \Lambda) \in X({^\vee}G^{\Gamma})$.   

We now turn to  $\tau$, which is a character of the component group of the
isotropy group  of $(y, \Lambda)$ in ${^\vee}G^{alg}$.  By relating
$\mathrm{Aut}({^\vee}G)$ to $\mathrm{Aut}(\Psi_{0}({^\vee}G))$
(Proposition 2.11 \cite{abv}) one may show that ${^\vee}\vartheta$ has
a unique lift to any finite algebraic cover of ${^\vee}G$.  The
collection of these lifts translates into a  lift of
${^\vee}\vartheta$ to an automorphism of ${^\vee}G^{alg}$.  We
abusively also denote this lift by ${^\vee}\vartheta$.
The image under ${^\vee}\vartheta$ of the isotropy group of
$(y,\Lambda)$ is equal to  the isotropy group of 
${^\vee}\vartheta^{\Gamma} \cdot (y, \Lambda)$.  This leads to the
definition of an action
\begin{equation}
\label{Xact}
{^\vee}\vartheta^{\Gamma} \cdot ((y,\Lambda), \tau) =
({^\vee}\vartheta^{\Gamma} \cdot (y,\Lambda),\ \tau \circ
{^\vee}\vartheta^{-1})
\end{equation}
in which the notation $\tau \circ {^\vee}\vartheta^{-1}$ is slightly
abusive because of the identifications we are making for ${^\vee}\vartheta$.  
One may verify that this action passes to equivalence classes.  It is
a simple exercise to prove that this action on equivalence classes is
insensitive 
to the choice of automorphism in $\mathrm{Int}({^\vee}G)\cdot {^\vee}\vartheta$
and to the choice of extension ${^\vee}\vartheta^{\Gamma}$.  In other words
the action on $\Xi(G/\mathbb{R})$ derived from $\vartheta^{\Gamma}$ is 
canonical.

Finally, let us return to the action of $\vartheta^{\Gamma}$ on
$\Pi(G/\mathbb{R})$.   For the complete picture, we must
consider how $\vartheta^{\Gamma}$ acts on canonical projective
representations of  strong real forms, \emph{i.e.} on $\Pi^{z}(G/\mathbb{R})$
for non-trivial $z \in Z({^\vee}G)^{\theta_{Z}}$ (Section
\ref{recall}).  In order to extend (\ref{repact}) to this setting, one
must lift $\vartheta$ to an automorphism of the canonical covering
$$1 \rightarrow \pi_{1}(G)^{can} \rightarrow G^{can} \rightarrow G
\rightarrow 1.$$
This lifting is achieved in  the same way as the lifting of
${^\vee}\vartheta$ to ${^\vee}G^{alg}$ above, except that one
restricts to \emph{distinguished} covers of $G$ (Definition 10.1
\cite{abv}) instead of finite algebraic covers.  We identify
$\vartheta$ with its lift to $G^{can}$ in (\ref{repact}) to define the
action of $\vartheta^{\Gamma}$ on $\Pi^{z}(G/\mathbb{R})$.

\section{Twisting and the local Langlands Correspondence}
\label{twistllc}

We continue with the hypotheses of Section \ref{actions1}, so that we
have compatible  automorphisms $\vartheta^{\Gamma}$ of $(G^{\Gamma},
\mathcal{W})$ and ${^\vee}\vartheta^{\Gamma}$ of (${^\vee}G^{\Gamma},
\mathcal{D})$.  Additionally, we have actions of $\vartheta^{\Gamma}$
on $\Pi^{z}(G/\mathbb{R})$, and of ${^\vee}\vartheta^{\Gamma}$ on
$\Xi({^\vee}G^{\Gamma})$.  Our goal is to prove that the local
Langlands Correspondence (\ref{llc3}) is
$\vartheta^{\Gamma}$-equivariant; \emph{i.e.}  that 
if $(\pi, \delta)$ corresponds to $((y,\Lambda), \tau)$ then
$\vartheta^{\Gamma} \cdot(\pi, \delta)$ corresponds to 
${^\vee}\vartheta^{\Gamma} \cdot ((y,\Lambda), \tau)$.

Our proof of this theorem proceeds in two steps.  In the first step
we prove the theorem when $G$ is an algebraic torus.  In this case it
is convenient to begin with the special case that the second invariant
$z \in Z(^{\vee}G)^{\theta_{Z}}$ is trivial, and  $Q =
\pi_{1}({^\vee}G)^{alg}$, before giving the proof in general.  

In the second step we make no assumptions on $G$ and use the first
step together with a classification of $\Pi^{z}(G/\mathbb{R})$ in
terms of tori (Chapter 11 \cite{abv}).

\subsection{The $\vartheta^{\Gamma}$-equivariance of the local Langlands
  Correspondence for tori}

In this section we assume $G = T$ is an algebraic torus.
The  local
Langlands Correspondence (\ref{llc2}) for tori is essentially proved
in Theorem 5.11 \cite{av92} (\emph{cf.} Proposition 10.6 and Corollary
10.7 \cite{abv}). We shall describe elements of this proof and apply
our automorphisms to them in order to arrive at the desired
equivariance. For the sake of simplicity  we begin under the
assumption that the second invariant
$z$ is trivial and $Q = \pi_{1}({^\vee}T)^{alg}$.  

We suppose that $(\pi,
\delta) \in \Pi(T/\mathbb{R})$  
corresponds to $((y,\Lambda), \tau) \in \Xi(T/\mathbb{R})$ under
(\ref{llc2}) and denote this correspondence by 
\begin{equation}
\label{llctori}
(\pi, \delta)  \leftrightarrow ((y,\Lambda),
\tau).
\end{equation}
As $T$ is abelian, there is only one inner form  of
$T$ in the inner class determined by $T^{\Gamma}$.  Thus, all
representations $\pi$ as above are representations of the real points
$T(\mathbb{R})$ of this single
inner form.  This allows us to divide correspondence (\ref{llctori})
into two separate correspondences
\begin{equation}
\label{llctori1}
\pi \leftrightarrow (y, \Lambda) \mbox{ and } \delta \leftrightarrow
\tau.
\end{equation}
We shall prove the $\vartheta^{\Gamma}$-equivariance of each of these
two correspondences in turn.

Observing again that  $T$ is abelian, we see that the canonical flat
$\Lambda$ is  just a singleton containing the differential 
$d\pi \in  \mathfrak{t}^{*} \cong {^\vee}\mathfrak{t}$ 
((9.3) and Proposition 10.6 \cite{abv}). Let $-a$ be the inverse map
composed with the 
automorphism of $X^{*}(T) \cong X_{*}({^\vee}T)$ obtained from
the first invariant of  $T^{\Gamma}$  (or ${^\vee}T^{\Gamma}$). 
The pair $(y,d \pi)$ corresponds to a unique element $\nu \in
X^{*}(T)/(1-a)X^{*}(T)$  
((4.7) \cite{av92}) whose values on the maximal compact subgroup of 
$T(\mathbb{R}) = T(\mathbb{R}, \delta)$ equal those of $\pi$.
Furthermore, the pair $(\nu, d\pi)$ defines a quasicharacter of
$T(\mathbb{R})$ equal to $\pi$ (Proposition 3.26 \cite{av92}).
In brief, we have correspondences
\begin{equation}
\label{3.26prop}
\pi \leftrightarrow  (\nu, d\pi) \leftrightarrow (y, d\pi)
\end{equation}
From this it is simple to   show that
$$\pi \circ \vartheta \leftrightarrow ( \nu \circ \vartheta , d\pi \circ 
d \vartheta)$$
when $\nu$ is regarded as an element of $X^{*}(T)$ and $d \pi$ is regarded as 
an element of $\mathfrak{t}^{*}$.  If $\nu$ is identified with an element
in $X_{*}({^\vee}T) \cong X^{*}(T)$, then $\nu \circ \vartheta$ is replaced with 
${^\vee}\vartheta \circ \nu$.  Similarly if $d\pi$ is
identified with an element in ${^\vee}\mathfrak{t} \cong \mathfrak{t}^{*}$, then
$d\pi \circ d\vartheta$ is replaced with $d{^\vee}\vartheta(d\pi)$ ((9.2)-(9.3)
\cite{abv}).  With these identifications in place, the correspondences
of (\ref{3.26prop}) prescribe that
\begin{equation}
\label{1corr}
\pi \circ \vartheta \leftrightarrow 
({^\vee}\vartheta \circ \nu, d {^\vee}\vartheta (d\pi)) 
\leftrightarrow {^\vee}\vartheta^{\Gamma}\cdot
(y,  d\pi) = {^\vee}\vartheta^{\Gamma}\cdot (y,\Lambda).
\end{equation}
The middle correspondence here requires a computation which
attaches a Langlands parameter to ${^\vee}\vartheta^{\Gamma}\cdot
(y,  d\pi)$ ((5.7) \cite{abv}), and follows the action of
${^\vee}\vartheta^{\Gamma}$ through  (4.6)(a)-(4.6)(b) and  (4.7)(c)
\cite{av92}.  The key  computation is in (4.7)(c)
\cite{av92}. It is of the form
\begin{align*}
& \frac{1}{2}(d {^\vee}\vartheta (d\pi) + a \circ d {^\vee}\vartheta
(d\pi)) -  (d {^\vee}\vartheta (X) -a \circ d {^\vee}\vartheta (X))\\
&= d {^\vee}\vartheta \left( \frac{1}{2}( d\pi + a
(d\pi)) -  ( X -a ( X))
\right)\\
& \leftrightarrow {^\vee}\vartheta \circ \nu
\end{align*}
for some $X \in {^\vee}\mathfrak{t}$.  It uses the commutativity of
$a$ and $d {^\vee}\vartheta$, which itself ensues from ${^\vee}\vartheta^{\Gamma}$
preserving $\mathcal{D}$.  
This completes our proof of the $\vartheta^{\Gamma}$-equivariance of the first
correspondence in (\ref{llctori1}). 

For the second correspondence in (\ref{llctori1})  set $x = (y, \Lambda)$
as above.  We shall use the
following three bijections with labelling from \cite{abv}, 
\begin{align*}
\{\mbox{equivalence classes of strong real forms}\} & 
\stackrel{(9.10) (b)}{\longleftrightarrow} T^{-\sigma, fin}/ T_{0}^{-\sigma}\\
&\stackrel{(9.10)(e)}{\cong} (X_{*}(T) \otimes_{\mathbb{Z}} \mathbb{Q})^{-a}/ 
(1-a)X_{*}(T)\\
& \stackrel{(9.8)(c)}{\cong} \hat{A}_{x}^{loc,alg}.
\end{align*}
We first describe
each of the three bijections, and then trace the action of $\vartheta^{\Gamma}$ 
through each
of them to obtain the desired correspondence between $(\vartheta^{\Gamma})^{-1}
(\delta)$
and $\tau \circ {^\vee}\vartheta^{-1}$.  

In the first bijection $\sigma = \sigma_{Z}$ is the antiholomorphic
involution of $T$ defined by $\mathrm{Int}(\delta)$.  The subgroup
$T^{-\sigma,fin}$ consists of elements $t \in T$ such that
$t\sigma(t)$ has finite order, and  $T_{0}^{-\sigma}$ is the identity
component of the subgroup consisting of elements such that
$t\sigma(t) = 1$.  Let $\delta_{0}$ be a strong
real form occurring in $\mathcal{W}$.  Then $\delta = t \delta_{0}$ for
some element $t \in T^{-\sigma,
  fin}$. The strong real form $\delta$ corresponds to the coset
$tT_{0}^{-\sigma}$ in the first bijection.  

 The
second bijection is induced 
from the map which sends 
$$\mu \in (X_{*}(T) \otimes_{\mathbb{Z}} \mathbb{Q})^{-a}
\cong (\mathfrak{t} \otimes \mathbb{Q})^{-a}$$
to $\exp(\pi i \mu) \in T^{-\sigma, fin}$ (Lemma 9.9
\cite{abv}). Here, the map $-a$ is the inverse map composed with the action of
the first invariant on $X_{*}(T)$.  Suppose
that $t = \exp( \pi i \mu)$ for such  a $\mu$ (with apologies for
the abusive notation for $\pi$ here).  

The third bijection is a  consequence of $X_{*}(T) \cong X^{*}({^\vee}T)$ and
its extension
$X^{*}({^\vee}T^{\mathrm{alg}}) \cong X_{*}(T) \otimes_{\mathbb{Z}} \mathbb{Q}$ ((9.7)
\cite{abv}).  
By the definition of the second correspondence in (\ref{llctori}), this
bijection sends $\mu$ to $\tau$.

Let us follow the action of $\vartheta^{\Gamma}$  through these three
bijections.  Since $\vartheta^{\Gamma}$ preserves the conjugacy class
$\mathcal{D}$ of $\delta_{0}$, we compute
$$(\vartheta^{\Gamma})^{-1}(\delta) = \vartheta^{-1}(t)
(\vartheta^{\Gamma})^{-1}(\delta_{0}) = \vartheta^{-1}(t) s \delta_{0}
s^{-1} = \vartheta^{-1}(t) s\sigma(s^{-1})\, \delta_{0}$$
for some $s \in T$.  The element $s \sigma(s^{-1})$ belongs to
$T_{0}^{-\sigma}$ (Proposition 9.10 \cite{abv}),  so
$(\vartheta^{\Gamma})^{-1}(\delta)$ 
corresponds to the element $\vartheta^{-1}(t)$ under the first
bijection. 
The latter element corresponds to $d\vartheta^{-1}(\mu)$ under the second
bijection.  The third bijection sends $d \vartheta^{-1} (\mu)$ to
$\tau \circ {^\vee}\vartheta^{-1}$.  In summary, the second
correspondence in (\ref{llctori1}) is $\vartheta^{\Gamma}$-equivariant, and
for tori we have the
desired correspondence
\begin{equation}
\label{descorr}
\vartheta^{\Gamma} \cdot (\pi, \delta) = 
(\pi\circ \vartheta, (\vartheta^{\Gamma})^{-1}
(\delta)) \leftrightarrow ({^\vee}\vartheta^{\Gamma} \cdot
(y,\Lambda), \tau \circ {^\vee} 
\vartheta^{-1}) = {^\vee}\vartheta^{\Gamma} \cdot ((y,\Lambda), \tau).
\end{equation}

We now discuss the proof of (\ref{llc2}) for tori allowing for the
possibility of non-trivial $z \in
Z({^\vee}T)^{\theta_{Z}} = {^\vee}T^{\theta_{Z}}$.  This does not
affect the second correspondence in (\ref{llctori1}) at all.
However, the first correspondence is affected and we must consider the
subspace
$${^\vee}\mathfrak{t}_{\mathbb{Q}} = 
X_{*}({^\vee}T) \otimes_{\mathbb{Z}} \mathbb{Q} \subset
X_{*}({^\vee}T) \otimes_{\mathbb{Z}} \mathbb{C} =
{^\vee}\mathfrak{t}.$$
 This subspace may be identified with a subspace
of $\mathfrak{t}^{*}$ using $X^{*}(T) \cong X_{*}({^\vee}T)$.  Suppose
$\lambda \in {^\vee}\mathfrak{t}_{\mathbb{Q}}$ satisfies $\exp(2 \pi i
\lambda) = z$.  According to Proposition 10.6 \cite{abv} and Theorem
5.11 \cite{av92} the correspondences of (\ref{3.26prop}) extend to 
the present setting if we allow $\nu$ to belong to  $\lambda +
X^{*}(T)$ (modulo $(1-a)X^{*}(T)$).  Having made this adjustment, one
may follow the action of $\vartheta^{\Gamma}$ through (\ref{3.26prop})
as before to obtain (\ref{1corr}) and the desired correspondence
(\ref{descorr}).

Finally, let us consider correspondence
(\ref{llc3}) for $\Xi^{z}({^\vee}T^{\Gamma})^{Q}$.   Recall that $Q$ is a
quotient of the algebraic fundamental group as in (\ref{qq}). Any element of
$\Xi^{z}({^\vee}T^{\Gamma})^{Q}$ may be equally regarded as an element  
$(x,\tau)  \in \Xi^{z}({^\vee}T^{\Gamma})$  in  which the restriction
 of $\tau$ to 
$$K_{Q} ({^\vee}G_{x}^{alg})_{0}/({^\vee}G_{x}^{alg})_{0} \subset
 A_{x}^{loc, alg}$$
is trivial.  The  automorphism ${^\vee}\vartheta^{\Gamma}$ acts on
$\Xi^{z}({^\vee}T^{\Gamma})^{Q}$  if and only if
${^\vee}\vartheta$ preserves this subgroup for all $(x,\tau)$
(via the lift of ${^\vee}\vartheta$).  In order to ensure this
property we assume that (the lift of) ${^\vee}\vartheta$ preserves
$K_{Q}$.  Equivalently, we assume that (the lift of)
${^\vee}\vartheta$ preserves $Q$ or $\hat{Q}$.  Under this assumption,
the $\vartheta^{\Gamma}$-equivariance of (\ref{descorr})
for 
$\Pi^{z}(T/\mathbb{R}) \stackrel{LLC}{\longleftrightarrow}\Xi^{z}({^\vee}T^{\Gamma})$
 restricts to 
$$\Pi^{z}(T/\mathbb{R})_{\hat{Q}} \stackrel{LLC}{\longleftrightarrow} 
\Xi^{z}({^\vee}T^{\Gamma})^{Q}.$$
The simplicity of this statement is slightly misleading in that we are
identifying the finite subgroup $J \subset T$, described at the end of
Section \ref{recall}, with $\hat{Q}$.  This identification is made
through the isomorphisms of (9.7)(d) and Lemma 9.9 (b) \cite{abv}.
It is straightforward to verify that these isomorphisms are
$\vartheta^{\Gamma}$-invariant in an appropriate sense.
We close by restating the above result as a theorem.
\begin{thm}
  \label{equivtori}
Suppose $T$ is an algebraic torus, $(T^{\Gamma}, \mathcal{W})$ is an
extended group for $T$, and $({^\vee}T^{\Gamma},\mathcal{D})$ is an E-group for $T^{\Gamma}$
with second invariant $z$. Suppose further that $\vartheta^{\Gamma}$
and ${^\vee}\vartheta^{\Gamma}$ are compatible automorphisms of
$(T^{\Gamma}, \mathcal{W})$ and  $({^\vee}T,\mathcal{D})$
respectively, and that  $Q$ is a ${^\vee}\vartheta$-stable  quotient
of $\pi_{1}({^\vee}T)^{alg}$ by a closed subgroup.  Then the bijection 
$$\Pi^{z}(T/\mathbb{R})_{\hat{Q}} \stackrel{LLC}{\longleftrightarrow} 
\Xi^{z}({^\vee}T^{\Gamma})^{Q}$$
of Theorem 10.11 \cite{abv} is $\vartheta^{\Gamma}$-equivariant in the
sense of (\ref{descorr}).
\end{thm}

\subsection{The $\vartheta^{\Gamma}$-equivariance of the local Langlands
  Correspondence for reductive groups}
\label{twistllc1}

We return to considering correspondence (\ref{llc3}) for arbitrary connected
reductive $G$.  As observed at the end of the previous section, the
$\vartheta^{\Gamma}$-equivariance of (\ref{llc3}) is an immediate
consequence of the $\vartheta^{\Gamma}$-equivariance of
correspondence (\ref{llc2}) under the assumption that the lift of
${^\vee}\vartheta$ preserves $Q$.  It suffices therefore to prove the
$\vartheta^{\Gamma}$-equivariance of (\ref{llc2}).  To this end,
suppose  $(G^{\Gamma}, \mathcal{W})$ is an
extended group for $G$, and $({^\vee}G,\mathcal{D})$ is an E-group for $G^{\Gamma}$
with second invariant $z$. We suppose further that $\vartheta^{\Gamma}$
and ${^\vee}\vartheta^{\Gamma}$ are compatible automorphisms of
$(G^{\Gamma}, \mathcal{W})$ and  $({^\vee}G,\mathcal{D})$
respectively.

Correspondence  (\ref{llc2})  is a combination of three bijections of
equivalence classes
\begin{equation}
\label{abbrev}
(\pi, \delta) \stackrel{(12.3)}{\longleftrightarrow} (\delta, \Lambda^{can}, 
R_{i \mathbb{R}}^{+}, R_{\mathbb{R}}^{+}) \stackrel{(13.13),
  (10.7)}{\longleftrightarrow}
((y,\Lambda_{1}), \tau_{1})  \stackrel{(12.9)}{\longleftrightarrow}
((y,\Lambda),\tau) 
\end{equation}
where the labels above the arrows refer to the requisite results in \cite{abv}. 
Our strategy for proving the $\vartheta^{\Gamma}$-equivariance of
(\ref{abbrev}) is to describe each of its three bijections and follow the
action of $\vartheta^{\Gamma}$ or ${^\vee}\vartheta^{\Gamma}$ through
each of them. 

We begin with the description of the first correspondence in
(\ref{abbrev})
\begin{equation}
\label{langlandsclass}
(\pi, \delta) \longleftrightarrow (\delta, \Lambda^{can}, 
R_{i \mathbb{R}}^{+}, R_{\mathbb{R}}^{+}).
\end{equation}
On the left we have $\delta$, (an equivalence class of) a strong real form,
 and $\pi$, an (equivalence class of an) irreducible 
representation of $G(\mathbb{R}, \delta)^{can}$.  On the right we have
(an equivalence class of) a $G$\emph{-limit character} (Definition
12.1 \cite{abv}).  Its definition begins with a Cartan subgroup $T^{\Gamma}
\subset G^{\Gamma}$ such that $\delta \in T^{\Gamma} - T$.  In this
circumstance $T(\mathbb{R}) = T(\mathbb{R}, \delta) \subset
G(\mathbb{R}, \delta)$ is a Cartan subgroup in the usual sense, and
has a canonical cover
as in (\ref{cancov}).  There is an element $z(\rho) \in
Z({^\vee}G)^{\theta_{Z}}$ defined  from the half-sum of 
the positive roots of $(G,T)$, which is independent of the choice of
positive system (Definition 4.9 \cite{abv}).   The $G$-limit character
in (\ref{langlandsclass})  is a
triple in which  $\Lambda^{can} \in \Pi^{z z(\rho)}(T(\mathbb{R}))$
has  differential $d \Lambda^{can} \in \mathfrak{t}^{*}$, and 
$R_{i\mathbb{R}}^{+}$ and $R_{\mathbb{R}}^{+}$ are  systems of positive
  imaginary and real roots respectively, satisfying  $\langle \alpha,
  d \Lambda^{can} \rangle \geq 0$ for all $\alpha
  \in R_{i \mathbb{R}}^{*}$ ((11.2) \cite{abv}).

Correspondence (\ref{langlandsclass}) is a version of the Langlands
classification for real reductive groups (\cite{langclass}).  We will
first consider a 
special case of this classification in which $\pi$ has regular
infinitesimal character and $d \Lambda^{can}$ is regular (Theorem
11.7 \cite{abv}).  Note however, that we are
working with canonical covers, which are merely pro-algebraic and not
algebraic as in Langlands' original classification.  
In this case, correspondence (\ref{langlandsclass}) may be described
by first attaching to $\Lambda = (\Lambda^{can}, 
R_{i \mathbb{R}}^{+}, R_{\mathbb{R}}^{+})$ a \emph{standard
representation} $M(\Lambda) = \mathrm{Ind}(P \uparrow G(\mathbb{R},
\delta)^{can}) 
(X_{L})$.  This is an induced  representation in which the real parabolic
subgroup  $P \subset G(\mathbb{R}, \delta)^{can}$ is determined by
$R_{i \mathbb{R}}^{+}$ and 
$R_{\mathbb{R}}^{+}$, and $X_{L}$ is a relative discrete series
representation of the Levi subgroup $L \subset P$ determined by
$R^{+}_{i\mathbb{R}}$ (see page 108 \cite{av92}).
The standard representation $M(\Lambda)$ has a unique irreducible
quotient $\pi(\Lambda)$, its \emph{Langlands quotient} (page 122
\cite{abv}).   Correspondence (\ref{langlandsclass}) is prescribed by
$\pi = \pi(\Lambda)$.  

Let us introduce the action of $\vartheta^{\Gamma}$ into
(\ref{langlandsclass}).  If $\pi$ is 
replaced by $\pi \circ \vartheta$ then $\pi \circ \vartheta =
\pi(\Lambda) \circ \vartheta$ is a representation of
$\vartheta^{-1}(G(\mathbb{R},\delta))= G(\mathbb{R}, 
(\vartheta^{\Gamma})^{-1} (\delta))$.   We wish to show that
$\pi(\Lambda) \circ \vartheta =   
\pi(\vartheta^{\Gamma} \cdot \Lambda)$ for some appropriate definition
of $\vartheta^{\Gamma} \cdot \Lambda$. Looking back to (\ref{repact}),
the obvious definition is
$$\vartheta^{\Gamma}\cdot \Lambda = \vartheta^{\Gamma} \cdot (\Lambda^{can}, R_{i
  \mathbb{R}}^{+}, R_{\mathbb{R}}^{+}) = ( \Lambda^{can} \circ \vartheta, 
 R_{i\mathbb{R}}^{+} \circ d\vartheta, R_{\mathbb{R}}^{+} \circ d\vartheta),$$
where $d\vartheta$ is the differential of $\vartheta$, and $\vartheta$
is identified with its lift to canonical covers as necessary.

Let us consider $\Lambda^{can} \circ \vartheta$.  Since $\vartheta(z)
= \vartheta(\delta_{0}^{2}) = z$ for any strong real form $\delta_{0}$
in $\mathcal{W}$, and $\vartheta(z(\rho)) =
z(\rho)$ (Definition 4.9 \cite{abv}), it is a formality to verify that
$\Lambda^{can} \circ 
\vartheta \in \Pi^{z z(\rho)}(\vartheta^{-1}(T(\mathbb{R})))$ and
corresponds to $X_{L} \circ \vartheta$.  It is
equally formal to verify that the parabolic subgroup
$\vartheta^{-1}(P) \subset G(\mathbb{R}, 
(\vartheta^{\Gamma})^{-1} (\delta))$ is determined by the 
positive root systems $R_{i\mathbb{R}}^{+} \circ d\vartheta$ and
$R_{\mathbb{R}}^{+} \circ d\vartheta$.  We deduce that 
$\pi(\Lambda) \circ \vartheta$ is the Langlands quotient of
\begin{equation}
  \label{indid}
M(\Lambda) \circ \vartheta  \cong  \mathrm{Ind}\left( \vartheta^{-1}(P)
\uparrow G(\mathbb{R}, 
(\vartheta^{\Gamma})^{-1}(\delta) )^{can} \right)(X_{L} \circ \vartheta) =
M(\vartheta^{\Gamma} \cdot \Lambda)
\end{equation}
(\emph{cf.} proof of Proposition 3.1 \cite{mezoajm}).  In other words
$\pi(\Lambda) \circ \vartheta = \pi(\vartheta^{\Gamma} \cdot \Lambda)$
and we are justified in writing
$$\vartheta^{\Gamma}\cdot (\pi, \delta) = (\pi(\Lambda) \circ \vartheta,
(\vartheta^{\Gamma})^{-1} (\delta))  \longleftrightarrow
((\vartheta^{\Gamma})^{-1}(\delta), \vartheta^{\Gamma} \cdot \Lambda) = 
\vartheta^{\Gamma} \cdot (\delta, \Lambda^{can}, 
R_{i \mathbb{R}}^{+}, R_{\mathbb{R}}^{+})$$
for  correspondence (\ref{langlandsclass}) in the case of regular
infinitesimal character.

The case of singular infinitesimal character necessitates some technical
conditions on $\Lambda = (\Lambda^{can}, R_{i\mathbb{R}}^{+},
R_{\mathbb{R}}^{+})$, defined as \emph{final} in Definition 11.13
\cite{abv}.  Beyond this, the principal difference from the regular
case is that $M(\Lambda)$ is now defined as $\mathrm{Ind}(P \uparrow
G(\mathbb{R}, \delta)^{can}) (\mathsf{\Psi}_{d\Lambda^{can} +
  \mu}^{d \Lambda^{can}}X_{L})$.  Here, $\mathsf{\Psi}_{d\Lambda^{can} +
  \mu}^{d \Lambda^{can}}$ is a Zuckerman translation functor in which 
$\mu \in \mathfrak{t}^{*}$ is a
strictly dominant weight with respect to some positive system
containing $R^{+}_{i\mathbb{R}}$ and $R^{+}_{\mathbb{R}}$ (Definition
5.1 \cite{spehvogan}\footnote{The superscript and subscript for
  $\mathsf{\Psi}$ are
  reversed in this reference.}).  The representation $X_{L}$ is in the relative
discrete series for the Levi subgroup $L$ and has infinitesimal
character $d\Lambda^{can} + \mu$.  The representation
$\mathsf{\Psi}_{d\Lambda^{can} +   \mu}^{d \Lambda^{can}}X_{L}$
belongs to the relative \emph{limit} of discrete series for $L$.  It
has infinitesimal character $d\Lambda^{can}$ and its distribution
character has a simple description given in Lemma 5.5
\cite{spehvogan}.  The $\vartheta^{\Gamma}$-equivariance of
(\ref{langlandsclass}) for singular infinitesimal character (Theorem
11.14 \cite{abv}) proceeds as in the regular case, but hinges on the
identity
$$\left( \mathsf{\Psi}_{d\Lambda^{can} +
  \mu}^{d \Lambda^{can}}X_{L} \right) \circ \vartheta \cong
\mathsf{\Psi}_{(d\Lambda^{can} +   \mu) \circ d\vartheta}^{d \Lambda^{can} \circ d\vartheta}
\left(X_{L} \circ \vartheta\right)$$
required in the modified version of (\ref{indid}).  This identity is
an immediate consequence of a comparison of the distribution
characters of both representations.  

Having dealt with the first correspondence of (\ref{abbrev}), we move
to the second, namely
\begin{equation}
  \label{Tcorr}
(\delta, \Lambda^{can}, 
R_{i \mathbb{R}}^{+}, R_{\mathbb{R}}^{+}) \longleftrightarrow
((y,\Lambda_{1}), \tau_{1}).
\end{equation}
As previously discussed,  $\Lambda^{can}$ is a representation of a
cover of  the real points $T(\mathbb{R}, \delta)$  of a Cartan
subgroup $T^{\Gamma} \subset G^{\Gamma}$.  Together with the additional
data $(\delta, R_{i \mathbb{R}}^{+}, R_{\mathbb{R}}^{+})$,
$T^{\Gamma}$ becomes a \emph{based Cartan subgroup} of $(G^{\Gamma}, \mathcal{W})$
(Definition 13.5 \cite{abv}). By Proposition 13.10 (a) \cite{abv}, this
based Cartan subgroup is paired with  a unique based Cartan subgroup
of $({^\vee}G^{\Gamma}, 
\mathcal{D})$ (Definition 13.7 \cite{abv}) up to conjugation.  Let
${^d}T^{\Gamma} \subset {^\vee}G^{\Gamma}$ be a Cartan subgroup in the
latter based Cartan subgroup.  The pairing allows us to view
${^d}T^{\Gamma}$ as an E-group for the extended group $(T^{\Gamma},
T\cdot \delta)$ (Definition 13.9 \cite{abv}).  In this view,
correspondence (\ref{Tcorr}) becomes correspondence (\ref{llctori})
(Proposition 13.12 \cite{abv}).  By Theorem \ref{equivtori} we may
write
$$\vartheta^{\Gamma} \cdot (\delta, \Lambda^{can}, 
R_{i \mathbb{R}}^{+}, R_{\mathbb{R}}^{+}) \longleftrightarrow
{^\vee}\vartheta^{\Gamma} \cdot ((y,\Lambda_{1}), \tau_{1}).$$

We may now proceed to the final correspondence of (\ref{abbrev}) which
is
\begin{equation}
  \label{3corr}
((y,\Lambda_{1}), \tau_{1})  \longleftrightarrow
((y,\Lambda),\tau).
\end{equation}
The proof of Theorem 12.9 \cite{abv} defines the map from right to
left in (\ref{3corr}).   The map from left to right is considerably
simpler for $\Lambda_{1}$.  The canonical flat $\Lambda_{1}$ is an element in
${^\vee}\mathfrak{t}$ which maps to the canonical flat $\Lambda =
\mathcal{F}(\Lambda_{1})$ as in Section \ref{recall}.  It is an
immediate consequence of the definitions that 
$${^\vee}\vartheta^{\Gamma} \cdot (y, \Lambda_{1}) =  
({^\vee}\vartheta^{\Gamma}(y), d\,{^\vee}\vartheta(\Lambda_{1}))
\mapsto ({^\vee}\vartheta^{\Gamma}(y),
d\,{^\vee}\vartheta(\mathcal{F}(\Lambda_{1}))
={^\vee}\vartheta^{\Gamma} \cdot (y, \Lambda).$$ 

As for $\tau_{1}$,
it passes to a representation of a quotient of $A^{loc,alg}_{(y, \Lambda_{1})}$
which is isomorphic to $A^{loc, alg}_{(y,\Lambda)}$ ((12.11)(e),
Definition 12.4 and Proposition 13.12 (a) \cite{abv}).  In this way
$\tau_{1}$ maps to a representation $\tau$ of $A^{loc,
  alg}_{(y,\Lambda)}$.  Moreover, the above isomorphism is derived
from a natural inclusion of component groups ((12.11)(e) and Lemma
12.10 (c) \cite{abv}) which is clearly ${^\vee}
\vartheta^{\Gamma}$-equivariant.  In particular, the representation
$\tau_{1} \circ {^\vee} 
\vartheta^{-1}$ passes to a quotient which maps to
the representation $\tau \circ {^\vee}\vartheta^{-1}$.  This completes
the ${^\vee}\vartheta^{\Gamma}$-equivariance of (\ref{3corr}) and also completes
final step in proving the  $\vartheta^{\Gamma}$-equivariance of (\ref{abbrev}).
\begin{thm}
\label{equivthm}
Suppose $G$ is a connected reductive  algebraic group, $(G^{\Gamma},
\mathcal{W})$ is an 
extended group for $G$, and $({^\vee}G^{\Gamma},\mathcal{D})$ is an
E-group for $G^{\Gamma}$ 
with second invariant $z$. Suppose further that $\vartheta^{\Gamma}$
and ${^\vee}\vartheta^{\Gamma}$ are compatible automorphisms of
$(G^{\Gamma}, \mathcal{W})$ and  $({^\vee}G,\mathcal{D})$
respectively, and that  $Q$ is a ${^\vee}\vartheta$-stable  quotient
of $\pi_{1}({^\vee}G)^{alg}$ by a closed subgroup.  Then the bijection 
$$\Pi^{z}(G/\mathbb{R})_{\hat{Q}} \stackrel{LLC}{\longleftrightarrow} 
\Xi^{z}({^\vee}G^{\Gamma})^{Q}$$
of Theorem 10.11 \cite{abv} is $\vartheta^{\Gamma}$-equivariant.
\end{thm}

\section{Twisted endoscopy}
\label{twistsec}

The goal of this section is to merge the notion of endoscopy and
endoscopic lifting as given in 26 \cite{abv} with that of twisted
endoscopy as given by 
Kottwitz and Shelstad (\cite{ks}).  We maintain the same 
hypotheses on $G^{\Gamma}$, ${^\vee}G^{\Gamma}$, $\vartheta$,
${^\vee}\vartheta^{\Gamma}$, etc. as in Section \ref{twistllc1}.

\subsection{Twisted endoscopic data}

The definition of endoscopic data in (2.1.1)-(2.1.4) \cite{ks} is made
relative to the group $G$, not its dual group.  By contrast,
Definition 26.15 \cite{abv} describes an  endoscopic datum relative
to any weak E-group.  We follow the latter point of view, which
applies to any weak E-group, despite our hypotheses.
\begin{definition}
\label{wed}
A \emph{weak endoscopic datum} for $({^\vee}G^{\Gamma},
{^\vee}\vartheta^{\Gamma}, Q)$
 is a pair $(s^{Q}, {^\vee}H^{\Gamma})$ in which
\begin{enumerate}
\item $s^{Q} \in {^\vee}G^{Q}$ and its image $s$ in ${^\vee}G$ is 
${^\vee}\vartheta$-semisimple (see (2.1.3) \cite{ks}).  

\item $ {^\vee}H^{\Gamma} \subseteq  {^\vee}G^{\Gamma}$ is a weak E-group for 
some connected reductive group $H$.  

\item ${^\vee}H^{\Gamma}$ is open in the fixed-point set of 
$\mathrm{Int}(s) \circ {^\vee}\vartheta^{\Gamma}$.
\end{enumerate}
An \emph{endoscopic datum} for $({^\vee}G^{\Gamma},
{^\vee}\vartheta^{\Gamma},Q)$ is a weak endoscopic datum as above
together with an  
${^\vee}H$-conjugacy class $\mathcal{D}_{H}$ of elements of finite order
in ${^\vee}H^{\Gamma} - {^\vee}H$,  whose inner automorphisms
preserve  splittings of ${^\vee}H$ (Proposition 2.11 \cite{abv}).
\end{definition}

\begin{exmp}
  Let us return to Example \ref{autex} in which we take
  $$\vartheta(x) = \tilde{J} \, (x^{\intercal})^{-1} \, \tilde{J}^{-1}, \ x
  \in \mathrm{GL}_{N}$$
and $Q$ to be trivial.  This is the backdrop of \cite{arthurbook} in which
Arthur provides  a class of examples for $s = s^{Q}$ in terms of a
matrix $J_{O,S}$ (1.2 \cite{arthurbook}).   To be more
specific, the matrix
$J_{O,S}$ is defined in terms of a pair of non-negative integers
$(N_{O},N'_{S})$ satisfying $N = N_{O}+ 2N_{S}'$ (pages 8-9
\cite{arthurbook}), and for
${^\vee}\vartheta = \vartheta$ the endoscopic element $s$ equals
$J_{O,S}^{-1} \tilde{J}$ (page 11 \cite{arthurbook}).  We extend
${^\vee}\vartheta$ to ${^\vee}\vartheta^{\Gamma}$ as in Proposition
\ref{sigmares1}.  

If $N_{O}$ is
zero or odd then there is only one weak
E-group ${^\vee}H^{\Gamma}$ satisfying requirements 2-3 of Definition
\ref{wed}.  It is isomorphic to the direct product of $\Gamma$ and
${^\vee}H = \mathrm{SO}_{N_{0}} \times \mathrm{Sp}_{2N_{S}'}$.  When $N_{O}$ is
non-zero and even then  ${^\vee}H$
remains the same.  However,  a non-trivial semidirect product with
$\Gamma$ is possible 
via an outer automorphism in the orthogonal group
$\mathrm{O}_{N_{O}}$, which is in the 
fixed-point set of   $\mathrm{Int}(s) \circ {^\vee}\vartheta^{\Gamma}$
(pages 10-11 \cite{arthurbook}).

In any case, ${^\vee}H^{\Gamma} \cong {^\vee}H \rtimes \Gamma$.  Let
$\gamma \in \Gamma$ be the non-identity element.  
There exists an element $h \in {^\vee}H$ such that
$\mathrm{Int}(h,\gamma)$ fixes a splitting.
One may therefore set $\mathcal{D}_{H}$ equal to the ${^\vee}H$-conjugacy
class of $(h,\gamma)$. When ${^\vee}H^{\Gamma} \cong {^\vee}H \times
\Gamma$ the element $h$ may  be chosen to equal the identity element.
Thus $(J_{O,S}^{-1} \tilde{J}, {^\vee}H \rtimes \Gamma,
  \mathcal{D}_{H})$ defines an  endoscopic datum for
  $(\mathrm{GL}_{N}, {^\vee}\vartheta^{\Gamma}, \{1\})$.

  This procedure exhausts all possibilities for endoscopic data
  attached to $s =J_{O,S}^{-1} \tilde{J}$.
To see this, recall from Section \ref{recall} that  the isomorphism
class of ${^\vee}H^{\Gamma}$ is equivalent to a pair of invariants
$(a, \bar{z})$, where $\bar{z} \in Z({^\vee}H)^{\mathrm{Int}(h,\gamma)}/ (1 +
\mathrm{Int}(h,z)) Z({^\vee}H)$. It is a simple exercise to show that
the fixed-point set 
$Z({^\vee}H)^{\mathrm{Int}(h,\gamma)}$ equals $Z({^\vee}H)$, and in
turn that $\bar{z}$ may be identified with an element in
$Z({^\vee}H)$.   According to Proposition 4.7 (c) \cite{abv}, this
element $\bar{z}$ determines a unique endoscopic datum.
\label{artexample}
\end{exmp}

\begin{exmp}
  \label{24ex}
A general source for endoscopic data  is outlined on page 24 \cite{ks}.
Suppose we are given a Langlands parameter $\phi: W_{\mathbb{R}}
\rightarrow {^\vee}G^{\Gamma}$ (Definition 5.2 \cite{abv}), a semisimple element
$s \in {^\vee}G$, and an automorphism ${^\vee}\vartheta^{\Gamma}$ of
${^\vee}G^{\Gamma}$.    Suppose further that
\begin{equation}
  \label{endsource}
\mathrm{Int}(s) \circ {^\vee}\vartheta^{\Gamma}
\circ \phi =   \phi.
\end{equation}
Then one may take $s^{Q} \in {^\vee}G^{Q}$ to be a lift of $s$ and
${^\vee}H^{\Gamma}$ 
to be the semidirect product of $\phi(W_{\mathbb{R}})$ and the identity
component of the fixed-point set of $\mathrm{Int}(s) \circ
{^\vee}\vartheta^{\Gamma}$ 
in ${^\vee}G$ (\emph{cf.} (26.16) \cite{abv}).  One may take
$\mathcal{D}_{H}$ to be the conjugacy class of $\phi(\gamma)$, where
$\gamma \in \Gamma$ is non-trivial.

We highlight a special case related to Example
\ref{cocycletwist}.  Suppose $\varsigma$ is trivial in that example.
Set 
${^\vee}\vartheta^{\Gamma} = \varsigma_{\mathsf{a}}$ for the
one-cocycle $\mathsf{a}$, and set  $s=1$.   Then (\ref{endsource}) reads as
$\phi =   \mathsf{a} \cdot \phi$.  This case pertains to L-packets fixed by
a central character attached to $\mathbf{a}$.
We note that this is more
restrictive than the twisting of  2 \cite{ks},  
where $\mathbf{a}$ may be taken  in $H^{1}(W_{\mathbb{R}},
Z({^\vee}G))$ not just $H^{1}(\Gamma, Z({^\vee}G))$.  For this reason our 
definition of twisted endoscopic data does not accommodate twisting by
arbitrary quasicharacters.
\end{exmp}

The next definition of an endoscopic group is identical to (26.17)(b)
\cite{abv} if one looks past the twisting data. 

\begin{definition}
Suppose 
$(s^{Q}, {^\vee}H^{\Gamma},
\mathcal{D}_{H})$ is an endoscopic datum for $({^\vee}G^{\Gamma}, 
{^\vee}\vartheta^{\Gamma},Q)$.   Then an \emph{extended endoscopic
group} is a pair $(H^{\Gamma}, \mathcal{W}_{H})$ in which $H^{\Gamma}$ is the
extended group whose first invariant equals that of ${^\vee}H^{\Gamma}$
and whose second invariant is trivial (see Section \ref{recall}).   
\end{definition}

\begin{exmp}
  \label{artexample1}
We continue with the endoscopic datum $(J_{O,S}^{-1}\tilde{J}, 
{^\vee}H \rtimes \Gamma, \mathcal{D}_{H})$ of
Example \ref{artexample}.
The first portion of its extended endoscopic group is 
 of the form $H^{\Gamma} = H \rtimes \Gamma$.  Here, $H =
 \mathrm{SO}_{N_{O}} \times \mathrm{SO}_{2N_{S}'+1}$ when $N_{0}$ is
 non-zero and
 even, and $H = \mathrm{Sp}_{N_{O}-1} \times \mathrm{SO}_{2N_{S}'+1}$
 when $N_{0}$ is odd.  The semidirect product here is
 determined by the semidirect product in ${^\vee}H \rtimes \Gamma$,
 since both semidirect products are determined by the first invariants.

The choice of a Whittaker datum $\mathcal{W}_{H}$ runs along the same
lines as the choice of $\mathcal{D}_{H}$ in Example \ref{artexample}.
One may choose an element $h \in H$ such that $(h,\gamma)$ is a strong 
real form which preserves a Borel subgroup $B$ (Proposition 2.7 and
Proposition 2.14 \cite{abv}).  We may set $\mathcal{W}_{H}$ equal to
the $H$-conjugacy class of $((h,\gamma), N, \chi)$, where $N$ is the
commutator subgroup of $B$ and $\chi$ is any non-degenerate unitary
character of the real points $N(\mathbb{R})$.
The proof of the uniqueness of this Whittaker datum follows the  proof of
the uniqueness of $\mathcal{D}_{H}$ in Example \ref{artexample} 
(Corollary 2.16 and Proposition 3.6 \cite{abv}).
 \end{exmp}

We record the natural
definition of equivalence of endoscopic data  following (26.15)
\cite{abv} and (2.1.6) \cite{ks}.
\begin{definition}
  \label{equivdata}
Two endoscopic data $(s^{Q}, {^\vee}H^{\Gamma}, \mathcal{D}_{H})$ and 
$(s_{1}^{Q}, {^\vee}H_{1}^{\Gamma}, \mathcal{D}_{H_{1}})$ for $({^\vee}G^{\Gamma}, 
{^\vee}\vartheta^{\Gamma},Q)$ are \emph{equivalent} if there 
exists $g^{Q} \in {^\vee}G^{Q}$ such that 
\begin{enumerate}
\item $g {^\vee}H_{1}^{\Gamma}g^{-1} = {^\vee}H^{\Gamma}$

\item $g \mathcal{D}_{H_{1}} g^{-1} = \mathcal{D}_{H}$

\item $g^{Q} s_{1}^{Q} \ {^\vee}\vartheta (g^{Q})^{-1} \in s^{Q}\, 
(Z({^\vee}H)^{\theta_{Z}, Q})_{0}$
\end{enumerate}
Here, $g\in {^\vee}G$ is the image of $g^{Q}$, ${^\vee}\vartheta$ is
the unique lift of ${^\vee}\vartheta^{\Gamma}$ to an automorphism of
$^{\vee}G^{Q}$, and
$(Z({^\vee}H)^{\theta_{Z}, Q})_{0}$ is the identity component of the
preimage of $Z({^\vee}H)^{\theta_{Z}}$ in ${^\vee}G^{Q}$.
\end{definition}

\subsection{Endoscopic lifting}

We begin with a description of the setting for twisted endoscopic
lifting.  We suppose that

\begin{itemize}
\item $\vartheta^{\Gamma}$ is a distinguished automorphism of an extended
  group $(G^{\Gamma}, \mathcal{W})$ such that $\vartheta =
  \vartheta^{\Gamma}_{|G}$ is of finite order.

  \item $^{\vee}\vartheta^{\Gamma}$ is
a distinguished automorphism of an E-group $({^\vee}G,
\mathcal{D})$ for 
$G$, which is compatible with $\vartheta^{\Gamma}$ and is of finite order.  

\item $Q$ is a quotient of $\pi_{1}(^{\vee}G)^{alg}$ by a closed
  subgroup.

\item $(s^{Q}, {^\vee}H, \mathcal{D}_{H})$ is an endoscopic datum for
  $({^\vee}G^{\Gamma},  {^\vee}\vartheta^{\Gamma},Q)$.

\item  $(H, \mathcal{W}_{H})$ is an extended endoscopic group for
  $(s^{Q}, {^\vee}H, \mathcal{D}_{H})$.  
\end{itemize}

The assumptions  of finite order
are new to this 
section and deserve some scrutiny.  What is actually at issue here is
that the restriction of $\vartheta$ to the centre $Z(G)$ must be of finite
order.  If an automorphism has this property, then it may be
composed with an inner automorphism so that it fixes a splitting and
is of finite order ((16.5)  \cite{humphreys}).  Furthermore, the map
$\vartheta \mapsto \Psi_{0}(\vartheta)$, used in Proposition
\ref{thetares1} and the constructions of Example \ref{autex}, is
insensitive to composition with inner automorphisms.  One may therefore
make such an adjustment without affecting the property of being an
automorphism of $(G^{\Gamma}, \mathcal{W})$ and without changing
${^\vee}\vartheta^{\Gamma}$.  In addition, such an adjustment has no
effect on the actions on the equivalence classes in
$\Pi^{z}(G/\mathbb{R})$.

This leaves us with the question of why one should assume that
$\vartheta$ is of finite order.  It is clear from the
discussion following the definition of compatibility (Definition
\ref{compdef}) that $\vartheta$ is of finite order if and only if
the distinguished automorphism ${^\vee}\vartheta =
{^\vee}\vartheta^{\Gamma}_{|\, {^\vee}G}$ is of 
finite order.  The latter condition is a fundamental requirement in
the machinery of 25 \cite{abv}  employed in endoscopic
lifting.  This is ultimately the justification for the finiteness assumption.

Although the automorphism ${^\vee}\vartheta^{\Gamma}$ is assumed to be
of finite order, the automorphism 
$\mathrm{Int}(s) \circ {^\vee}\vartheta^{\Gamma}$ used in the
definition of endoscopic data need not be.  This also gets
in the way of using  25 \cite{abv}.   We may as well 
take care of this problem now.

This problem occurs in standard endoscopy too
and may be circumvented by replacing $s^{Q}$ with an element
$(s^{Q})'$ of finite order ((26.21) \cite{abv}).  This circumvention is made
possible by Lemma 26.20 \cite{abv}, which relies on the semisimplicity
of $s$.  We need an analogue of this semisimplicity for $\mathrm{Int}(s)
\circ {^\vee}\vartheta^{\Gamma}$.
\begin{lem}
$(s,{^\vee}\vartheta^{\Gamma})$  is a semisimple element in the
algebraic group ${^\vee}G^{\Gamma} \rtimes \langle {^\vee}\vartheta^{\Gamma}\rangle$. 
\end{lem}
\begin{proof}
The obvious semidirect product ${^\vee}G^{\Gamma} \rtimes \langle
{^\vee}\vartheta^{\Gamma}\rangle$ is an algebraic group as
 ${^\vee}G^{\Gamma}$ is an algebraic group,
${^\vee}\vartheta^{\Gamma}$ is an algebraic morphism, and
${^\vee}\vartheta^{\Gamma}$ is of finite 
order.   The
first property of Definition 5.1 is equivalent to $(s,
{^\vee}\vartheta^{\Gamma})$ being a semisimple element in ${^\vee}G \rtimes \langle
{^\vee}\vartheta^{\Gamma} \rangle$ (7 \cite{steinberg}).
\end{proof}
We may now apply Lemma 26.20 \cite{abv} to the  disconnected algebraic group
${^\vee}G^{\Gamma} \rtimes \langle {^\vee}\vartheta^{\Gamma} \rangle$ and the
semisimple element $(s, {^\vee}\vartheta^{\Gamma})$ to conclude that
there is an element $(s', {^\vee}\vartheta^{\Gamma}) \in (s,
{^\vee}\vartheta^{\Gamma}) \,Z({^\vee}H^{\Gamma})_{0}$ of finite order,
and that ${^\vee}H^{\Gamma}$ is open in the fixed-point set of
$\mathrm{Int}(s') \circ {^\vee}\vartheta^{\Gamma}$.
In particular,  $\mathrm{Int}(s') \circ
{^\vee}\vartheta^{\Gamma}$ is a finite-order automorphism and $s' \in
s Z({^\vee}H)^{\theta_{Z}}_{0}$.   There exists a 
lift $(s^{Q})' \in {^\vee}G^{Q}$ of $s'$ such that $(s^{Q})' \in
s^{Q} (Z({^\vee}H)^{\theta_{Z}, Q})_{0}$ and by Definition
\ref{equivdata} $((s^{Q})', {^\vee}H^{\Gamma}, \mathcal{D}_{H})$ is an
equivalent endoscopic datum with the desired finiteness property
(\emph{cf.}  (26.21) \cite{abv}).  We
may and shall assume from now on that $s^{Q} = (s^{Q})'$.

Calling to mind the setting for this section once more, let
$$\epsilon:{^\vee}H^{\Gamma} \rightarrow {^\vee}G^{\Gamma}$$
be the  inclusion
map.  The map $\epsilon$ induces several other maps on sets of
objects we have recalled in Section \ref{recall}.  For example there
is the map $X(\epsilon) : X({^\vee}H^{\Gamma}) \rightarrow
X({^\vee}G^{\Gamma})$ (Corollary 6.21 \cite{abv}) defined by
\begin{equation}
  \label{xep}
X(\epsilon)(y, \Lambda) = (\epsilon(y), \mathcal{F}\circ
d\epsilon(\Lambda)), \ (y,\Lambda) \in X({^\vee}H^{\Gamma}).
\end{equation}
There is also the map $\epsilon^{\bullet}: {^\vee}H^{Q_{H}}
\rightarrow {^\vee}G^{Q}$ ((5.14)(c) \cite{abv}) in which $Q_{H} = Q
\cap ({^\vee}H^{Q})_{0}$ ((26.1)(c) \cite{abv}).  The pair of maps
$(X(\epsilon), \epsilon^{\bullet})$ transfers the
${^\vee}H^{Q_{H}}$-action on $X({^\vee}H^{\Gamma})$ to the
${^\vee}G^{Q}$-action on $X({^\vee}G^{\Gamma})$, and therefore induces
a map of the corresponding orbits ((7.17)(c) \cite{abv}).  Similarly,
the pair $(X(\epsilon), \epsilon^{\bullet})$ induces homomorphisms
between isotropy subgroups and their component groups
$$A^{loc}(\epsilon): A^{loc, Q_{H}}_{(y, \Lambda)} \rightarrow A^{loc,
  Q}_{X(\epsilon)(y,\Lambda)}$$ 
((7.19)(c) \cite{abv}).
\subsection{Standard endoscopic lifting}
\label{standendolift}

To explain endoscopic lifting, we must now describe how $\epsilon$ behaves
on the level of sheaves on $X({^\vee}G^{\Gamma})$.  We will assume
that the reader has some familiarity with constructible and perverse
sheaves, and review some important facts provided in \cite{abv}.

The extension by zero map ((7.10)(c) \cite{abv})
$$\mu : \Xi({^\vee}G^{\Gamma})^{Q} \rightarrow
\mathrm{Ob}\, \mathcal{C}(X({^\vee}G^{\Gamma}), {^\vee}G^{Q})$$
yields a bijection between the  complete geometric parameters of
type $Q$ and the irreducible  ${^\vee}G^{Q}$-equivariant constructible
sheaves of complex vector spaces on $X({^\vee}G^{\Gamma})$.    The
perverse extension map ((7.10)(d) \cite{abv}, Definition 1.4.22 \cite{bbd})
$$P:  \Xi({^\vee}G^{\Gamma})^{Q} \rightarrow
\mathrm{Ob}\, \mathcal{P}(X({^\vee}G^{\Gamma}), {^\vee}G^{Q})$$
yields a bijection between the  complete geometric parameters of
type $Q$ and the irreducible  ${^\vee}G^{Q}$-equivariant perverse
sheaves of complex vector spaces on $X({^\vee}G^{\Gamma})$.  Let
$K\mathcal{C}(X({^\vee}G^{\Gamma}), {^\vee}G^{Q})$ and $K
\mathcal{P}(X({^\vee}G^{\Gamma}), {^\vee}G^{Q})$ denote the
Grothendieck groups of the categories
$\mathcal{C}(X({^\vee}G^{\Gamma}), {^\vee}G^{Q})$ and
$\mathcal{P}(X({^\vee}G^{\Gamma}), {^\vee}G^{Q})$ respectively.  There
is an isomorphism $\chi : K\mathcal{P}(X({^\vee}G^{\Gamma}),
{^\vee}G^{Q}) \rightarrow K\mathcal{C}(X({^\vee}G^{\Gamma}),
{^\vee}G^{Q})$ defined by
\begin{equation}
  \label{chimap}
\chi(P) = \sum_{i} (-1)^{i} H^{i}P, \ P \in \mathrm{Ob}
\mathcal{P}(X({^\vee}G^{\Gamma}), {^\vee}G^{Q})
\end{equation}
(Lemma 7.8 \cite{abv}) which allows one to identify the
two Grothendieck groups.

Correspondence (\ref{llc3})
furnishes  a natural perfect pairing 
\begin{equation}
  \label{pp}
\langle \, , \, \rangle : K \Pi^{z}(G/\mathbb{R})_{\hat{Q}} \times
K\mathcal{C}(X({^\vee}G^{\Gamma}), {^\vee}G^{Q}) \rightarrow
\mathbb{Z}
\end{equation}
with the Grothendieck group of the category of finite length
representations of strong real forms of $G$ of type $\hat{Q}$ (Definition 15.11
\cite{abv}).  In essence, the irreducible constructible sheaves
$\mu(\xi)$ are paired with characters of standard
  representations $M(\xi)$ for $\xi \in \Xi({^\vee}G^{\Gamma})^{Q}$. This
  pairing is equivalent to another perfect pairing, 
which we abusively denote the same way,  
\begin{equation}
  \label{pp1}
\langle \, ,\,  \rangle : K \Pi^{z}(G/\mathbb{R})_{\hat{Q}} \times
K\mathcal{P}(X({^\vee}G^{\Gamma}), {^\vee}G^{Q}) \rightarrow
\mathbb{Z}
\end{equation}
via the isomorphism $\chi$.  Using the Kazhdan-Lusztig-Vogan algorithm, this
second pairing is shown to be a simple pairing between irreducible
perverse sheaves $P(\xi)$  and  characters of irreducible
representations $\pi(\xi)$ for $\xi \in \Xi({^\vee}G^{\Gamma})^{Q}$
((11.2)(e), Theorem  15.12 and  
Theorem 26.2 \cite{abv}).

We are now in the position to describe endoscopic lifting in the
standard case,  \emph{i.e.} when  ${^\vee}\vartheta^{\Gamma}$ is
trivial.
It is essentially given by the restriction homomorphism
\begin{equation}
  \label{epstar}
\epsilon^{*} : K\mathcal{C}(X({^\vee}G^{\Gamma}), {^\vee}G^{Q})
\rightarrow K\mathcal{C}(X({^\vee}H^{\Gamma}), {^\vee}H^{Q_{H}})
\end{equation}
which is defined by  the inverse image functor on equivariant constructible
sheaves (Proposition 7.18 and (7.19)(d) \cite{abv}).  The usual notion
of endoscopic lifting between linear combinations of characters may be
recovered by combining 
$\epsilon^{*}$ with the prefect pairing (\ref{pp}) as follows.  Let
$\eta_{H}$ be a strongly stable complex linear combination of characters of 
representations in $\Pi^{z_{H}}(H/\mathbb{R})_{\hat{Q}_{H}}$
(\emph{cf.} Definition 18.6 and Definition 26.13 \cite{abv}). 
Then (\ref{pp}) allows us to view $\eta_{H}$ as a complex-valued
homomorphism on $K\mathcal{C}(X({^\vee}H^{\Gamma}), {^\vee}H^{Q_{H}})$
(Theorem 26.2 \cite{abv}). Taken this way, endoscopic lifting is
defined by $\epsilon_{*}$, where 
\begin{equation}
  \label{endolift}
\epsilon_{*}\eta_{H}( F_{G}) = \langle \epsilon_{*}\eta_{H},
F_{G}\rangle= \langle \eta_{H}, 
\epsilon^{*}F_{G} \rangle , \ F_{G}
\in K\mathcal{C}(X({^\vee}G^{\Gamma}), {^\vee}G^{Q})
\end{equation}
(Definition 26.3 and Definition 26.18 \cite{abv}).  One may regard
$\epsilon_{*}\eta_{H}$ as a homomorphism on
$K\mathcal{C}(X({^\vee}G^{\Gamma}), {^\vee}G^{Q})$ or, somewhat more
traditionally,  as a formal
complex combination of characters of representations in $\Pi^{z}(G/\mathbb{R})_{\hat{Q}}$.

We have just described endoscopic lifting in terms of constructible
sheaves and the pairing (\ref{pp}).  In (\ref{endolift}) one should
regard $\eta_{H}$ as a 
linear combination of characters of standard representations. We could
equally well have 
described endoscopic lifting in terms of perverse sheaves and pairing
(\ref{pp1}).     To do this,  one should regard $\eta_{H}$ as a
linear combination of 
irreducible characters.

\subsection{Twisted endoscopic lifting}
\label{twistendolift}

We now consider \emph{twisted} endoscopic lifting,
\emph{i.e.} when  
${^\vee}\vartheta^{\Gamma}$ is a non-trivial outer automorphism.  In the twisted case not all
irreducible elements of  $K\mathcal{C}(X({^\vee}G^{\Gamma}),
{^\vee}G^{Q})$ are relevant in (\ref{epstar}).
This is described in 25 
\cite{abv}, although in much less detail than the case of standard
endoscopy.   For consistency with 25 
\cite{abv}, set $\sigma = 
\mathrm{Int}(s) \circ {^\vee}\vartheta^{\Gamma}$, keeping in mind that
this is an automorphism of order $m < \infty$.  We note that the
automorphism $\sigma$ acts 
on ${^\vee}G^{Q}$ (through  a unique lift).  Furthermore it acts  on
$X({^\vee}G^{\Gamma})$ ((\ref{Xact})) in a manner that is compatible
with the ${^\vee}G^{Q}$-action on $X({^\vee}G^{\Gamma})$ ((25.1)(b)
\cite{abv}).  It follows that $\sigma$ permutes the
${^\vee}G^{Q}$-orbits of $X({^\vee}G^{\Gamma})$, and induces an isomorphism
$A^{loc,Q}_{x} \mapsto A^{loc,Q}_{\sigma\cdot x}$
for any $x \in X({^\vee}G^{\Gamma})$.  We first provide an
example of the type of sheaves which are relevant to (\ref{epstar}) in
twisted endoscopy.

\begin{exmp}
Suppose $x \in X({^\vee}G^{\Gamma})$ and $\sigma(x)= x$.  Suppose
further that  $S
= {^\vee}G^{Q} \cdot x$ and $\tau = \sum_{j=1}^{k} \tau_{j}$ is  a sum of
irreducible representations of $A_{x}^{loc,Q}$.  The automorphism
$\sigma$ induces an automorphism of 
$A_{x}^{loc,Q}$.  We suppose that
$\tau \circ \sigma \cong \tau$.  Under these circumstances $\tau$ extends
to a representation of $A_{x}^{loc, Q} \rtimes \langle \sigma
\rangle$, where $\tau(\sigma)$ is an intertwining operator between
$\tau \circ \sigma $ and $\tau$.  Note that the choice for
$\tau(\sigma)$ is not canonical as $\tau(\sigma)$ may be replaced by
$\zeta \tau(\sigma)$ for any $m$th root of unity $ \zeta \in
\mathbb{C}^{\times}$.

We wish to translate this setup into the
realm of constructible sheaves.  Let $\xi_{j} = (S, \tau_{j})$.
The pair $(S,\tau)$ is equivalent to a 
${^\vee}G^{Q}$-equivariant local system
$\oplus_{j=1}^{k}\mathcal{V}_{\xi_{j}}$ on  $S$ 
((7.4) and Lemma 7.3 (c) \cite{abv}).  The automorphism $\sigma$ passes to an
automorphism of $\oplus_{j=1}^{k} \mathcal{V}_{\xi_{j}}$ in the guise of
$\tau(\sigma)$,   a linear isomorphism of the stalks.  This
automorphism of $\oplus_{j=1}^{k} \mathcal{V}_{\xi_{j}}$ further
passes to an automorphism of the 
${^\vee}G^{Q}$-equivariant 
constructible sheaf $\oplus_{j=1}^{k} \mu(\xi_{j})$  through the
same isomorphism of the stalks.  We abusively denote the isomorphism
of $\oplus_{j=1}^{k} \mu(\xi_{j})$ by 
$\tau(\sigma)$.  The intertwining property of
$\tau(\sigma)$ above is tantamount to a compatibility
condition with the $\sigma$-action on ${^\vee}G^{Q}$.  The pairs 
$(\oplus_{j=1}^{k} \mu(\xi_{j}), \tau(\sigma))$, become the principal objects in
twisted endoscopy.
\label{sfix}
\end{exmp}

A slight generalization of the sheaves of Example \ref{sfix} is given  in (25.7)
\cite{abv}.  There the objects of the category
$\mathcal{C}(X({^\vee}G^{\Gamma}), {^\vee}G^{Q}, \sigma)$ 
are pairs $(C, \sigma_{C})$, in which $C$ is a
${^\vee}G^{Q}$-equivariant constructible sheaf, and $\sigma_{C}$ is
a finite-order automorphism of $C$ compatible with the
$\sigma$-action on ${^\vee}G^{Q}$ and $X({^\vee}G^{\Gamma})$.  The
compatibility condition is natural, but a bit unwieldy to state.  As
it is not given in \cite{abv}, we provide it here:
Given any $x \in X({^\vee}G^{\Gamma})$, $S = {^\vee}G^{Q} 
\cdot x$ and stalk  $C_{x}$, there is a representation
$\tau_{S}^{loc}(C)$ of $A_{S}^{loc,Q}$ with space $C_{x}$, obtained by
restricting $C$ to $S$ (Corollary 23.3
\cite{abv}). The compatibility condition on the  
morphism $\sigma_{C}: C \rightarrow C$
is that it induces a linear isomorphism $C_{x}
\rightarrow C_{\sigma \cdot x}$ intertwining $\tau^{loc}_{S}(C)$ with
$\tau^{loc}_{\sigma(S)}(C)$ for every $x$.
In the twisted case the restriction homomorphism (\ref{epstar}) of
standard endoscopy is
replaced with
\begin{equation}
  \label{epstar1}
\epsilon^{*} : K\mathcal{C}(X({^\vee}G^{\Gamma}), {^\vee}G^{Q}, \sigma)
\rightarrow K\mathcal{C}(X({^\vee}H^{\Gamma}), {^\vee}H^{Q_{H}},
\sigma).
\end{equation}
Despite appearances, the Grothendieck group
$K\mathcal{C}(X({^\vee}H^{\Gamma}), {^\vee}H^{Q_{H}}, 
\sigma)$ on the right differs only superficially from
$K\mathcal{C}(X({^\vee}H^{\Gamma}), {^\vee}H^{Q_{H}})$.  Indeed,
$\sigma$ acts trivially on all objects defined from ${^\vee}H$
(\emph{cf.} Definition \ref{wed}).   In consequence,
all of the intertwining operators on stalks may be identified with
their eigenvalues, which are $m$th roots of unity.  This leads to an
isomorphism
\begin{equation}
  \label{tensiso}
K\mathcal{C}(X({^\vee}H^{\Gamma}), {^\vee}H^{Q_{H}}, \sigma) \cong
K\mathcal{C}(X({^\vee}H^{\Gamma}), {^\vee}H^{Q_{H}}) \otimes \mathbb{Z}[U_{m}]
\end{equation}
where $\mathbb{Z}[U_{m}] \subset \mathbb{C}$ is the group algebra of
the $m$th roots of unity (\emph{cf.}
(25.7)(e) \cite{abv}). 

It is worth reflecting on  the relationship between (\ref{epstar})
and (\ref{epstar1}).  In the
case of standard endoscopy,  $\sigma$ is simply 
$\mathrm{Int}(s)$, and $\sigma \cdot x \in S = {^\vee}G^{Q} \cdot x$
for every $x \in X({^\vee}G^{\Gamma})$. The action of $\sigma$ falls
under the umbrella of the equivariance conditions on $C
\in \mathcal{C}(X({^\vee}G^{\Gamma}), {^\vee}G^{Q})$.  In particular
the equivariance produces an
identification of stalks $C_{x} = C_{\sigma \cdot x}$,
and the identity map
on $C_{x}$  is the canonical choice of intertwining operator between
$\tau^{loc}_{S}(C)$ and $\tau^{loc}_{\sigma(S)}(C)$.  In
other words, one may define $\sigma_{C}$ by taking it to be the
identity map on all stalks.  Roots of unity still play a role in the
definition of $\mathcal{C}(X({^\vee}G^{\Gamma}), {^\vee}G^{Q},
\sigma)$ as in Example \ref{sfix}, but they play a  superficial role as
in $\mathcal{C}(X({^\vee}H^{\Gamma}), {^\vee}H^{Q_{H}},
\sigma)$.  In summary, for standard endoscopy
(\ref{epstar1}) is of the form
\begin{equation}
  \label{standend}
\epsilon^{*} : K\mathcal{C}(X({^\vee}G^{\Gamma}), {^\vee}G^{Q})
\otimes \mathbb{Z}[U_{m}]
\rightarrow K\mathcal{C}(X({^\vee}H^{\Gamma}), {^\vee}H^{Q_{H}},
\sigma) \otimes \mathbb{Z}[U_{m}],
\end{equation}
where $\epsilon^{*}$ is the identity map on $U_{m}$.
This is the trivial extension of (\ref{epstar}) to the tensor
products with $\mathbb{Z}[U_{m}]$.

By contrast, in the case of twisted endoscopy, $\sigma$ is an outer
automorphism and the equivariance of $C$ does not necessarily allow us to
identify a stalk $C_{x}$ with $C_{\sigma \cdot x}$.  Even in the case
that $\sigma \cdot x = x$ as in Example \ref{sfix} there need not exist
a canonical choice of intertwining operator on $C_{x}$.

In order to make a connection with Arthur packets in the next
section, it is preferable to express twisted endoscopic lifting using
perverse sheaves rather than constructible sheaves.
The obvious extension of the foregoing discussion to chain complexes
of equivariant constructible sheaves is 
defined by specifying that the sheaves in each degree of a complex belong to
$\mathcal{C}(X({^\vee}G^{\Gamma}), {^\vee}G^{Q}, \sigma)$.  One
defines the category of perverse sheaves
$\mathcal{P}(X({^\vee}G^{\Gamma}), {^\vee}G^{Q}, \sigma)$ using the
same principle.  We call an object in
$\mathcal{P}(X({^\vee}G^{\Gamma}), {^\vee}G^{Q}, 
\sigma)$ a \emph{twisted perverse sheaf}.  Such an object  is a  pair
$(P,\sigma_{P})$, where $P = P^{\bullet}$ is 
a constructible complex (in the derived category and with additional
structure), and 
$\sigma_{P}$  is an automorphism  of $P$.
It is easily verified that $\sigma_{P}$ induces an automorphism of
$H^{i}P$ for all $i \in \mathbb{Z}$.
In consequence, one may extend (\ref{chimap}) to an isomorphism
\begin{equation}
  \label{chimap1}
  \chi: K \mathcal{P}(X({^\vee}G^{\Gamma}), {^\vee}G^{Q},
\sigma) \rightarrow K\mathcal{C}(X({^\vee}G^{\Gamma}), {^\vee}G^{Q},
\sigma)
\end{equation}
in which $\chi(\sigma_{P}) $  is the induced  morphism on
$\oplus_{i \in \mathbb{Z}} H^{i} P$.
We abusively denote the restriction homomorphism
$$\epsilon^{*} : K \mathcal{P}(X({^\vee}G^{\Gamma}), {^\vee}G^{Q},
\sigma) \rightarrow  K \mathcal{P}(X({^\vee}H^{\Gamma}), {^\vee}H^{Q_{H}},
\sigma)$$
as in (\ref{epstar1}).  The codomain here may be simplified as in
(\ref{tensiso}).

Our next step is to describe twisted endoscopic
lifting as a map on stable virtual characters along the lines of
(\ref{endolift}).  Rather than working with arbitrary stable virtual
characters $\eta_{H}$ as in (\ref{endolift}) one may work, without loss
of generality, with a
specific family of strongly stable virtual characters
$\eta_{S}^{mic, Q_{H}}$, where $S$ runs over all ${^\vee}H$-orbits of
$X({^\vee}H^{\Gamma})$ (Corollary 19.16 \cite{abv}).  The definition
of $\eta_{S}^{mic,Q_{H}}$ rests on deep results in microlocal geometry which
we shall not attempt to sketch in any detail.  We shall settle for a
peripheral  description of
$\eta_{S}^{mic,Q_{H}}$  in terms of representations $\tau^{mic}_{S}(P)$,
which are  analogues of the representations
$\tau_{S}^{loc}(C)$ mentioned earlier.  The
representation $\tau_{S}^{loc}(C)$ 
of $A^{loc,Q_{H}}_{S}$ defines the local system equal to the restriction of $C
\in \mathcal{C}(X( {^\vee}H^{\Gamma}), {^\vee}H^{Q_{H}})$
to $S$.  Similarly, given $P \in \mathcal{P}(X( {^\vee}H^{\Gamma}),
{^\vee}H^{Q_{H}})$ there exists a representation $\tau^{mic}_{S}(P)$ of the
\emph{equivariant micro-fundamental group} $A_{S}^{mic,Q_{H}}$ which defines
a ${^\vee}H^{Q_{H}}$-equivariant local system $Q^{mic}(P)$ on a
space determined by ${^\vee}H^{Q_{H}}$ and $X({^\vee}H^{\Gamma})$
((24.1), Definition 24.7, Theorem 24.8 and 
Corollary 24.9 \cite{abv}).   Passing over the  sophisticated theory
underlying these objects, we may write
$$\eta_{S}^{mic,Q_{H}} =  \sum_{\xi \in \Xi({^\vee}H)^{Q_{H}}} e(\xi)
(-1)^{\dim S_{\xi} - \dim S }\ 
\dim \tau_{S}^{mic}(P(\xi))\  \pi(\xi)$$
(Definition 19.13, Definition 19.15, Corollary 19.16 and Corollary
24.9 (a) \cite{abv}).
This is a finite sum in which $e(\xi)  = \pm 1$ (Definition 15.8
\cite{abv}), $S_{\xi}$ is the ${^\vee}H$-orbit in $\xi$ ((7.4)
\cite{abv}), and $\pi(\xi)$ is an irreducible (character of a)
representation in $\Pi^{z_{H}}(H/\mathbb{R})_{\hat{Q}_{H}}$
given by (\ref{llc3}).    More generally, for any $h \in A_{S}^{mic,Q_{H}}$ one
may define a formal complex virtual representation
\begin{equation}
  \label{etadef}
  \eta_{S}^{mic,Q_{H}}(h) =  \sum_{\xi \in \Xi({^\vee}H)^{Q_{H}}} e(\xi)
(-1)^{\dim S_{\xi} - \dim S }\ \mathrm{tr}\left(
  \tau_{S}^{mic}(P(\xi))(h)\right) \,  \pi(\xi)
\end{equation}
(Definition 26.8 \cite{abv}).
According to Lemma 26.9 \cite{abv},
\begin{equation}
  \label{etapair}
\langle \eta_{S}^{mic,Q_{H}}(h), P(\xi') \rangle = (-1)^{\dim S}
\, \mathrm{tr}\left(
  \tau_{S}^{mic}(P(\xi'))(h)\right), \ \xi' \in
  \Xi({^\vee}H^{\Gamma})^{Q_{H}}
  \end{equation}
using pairing (\ref{pp1}).

We would like  to define an extension of $\eta_{S}^{mic, Q_{H}}$ to 
$A_{S}^{mic,Q_{H}}  \times \langle \sigma \rangle$ in order to
incorporate the twisted perverse sheaves of $\mathcal{P}(X({^\vee}H^{\Gamma}), 
{^\vee}H^{Q_{H}}, \sigma)$. Suppose $\xi \in \Xi({^\vee}H^{\Gamma})^{Q_{H}}$ and
 $(P(\xi), \sigma_{P(\xi)}) \in \mathcal{P}(X({^\vee}H^{\Gamma}),
{^\vee}H^{Q_{H}}, \sigma)$. We first note that
$\sigma_{P(\xi)}$ may be identified with an $m$th root of unity.
This identification is the essence of the notion of ``eigenobjects''
on page 272 \cite{abv}, which goes as follows:
Since $\sigma$ fixes all $x \in X({^\vee} H^{\Gamma})$, the automorphism
$\sigma_{P(\xi)}$ must induce a linear 
automorphism on each stalk  $H^{i}P(\xi)_{x}$.  Let $\zeta \in \mathbb{C}$ be an
eigenvalue on some stalk $H^{i}P(\xi)_{x}$ and let $\zeta_{P(\xi)}$ be the
automorphism of $P(\xi)$ given by multiplication by $\zeta$ on every
stalk.  Clearly, $\zeta \in U_{m}$  and
$\ker (\sigma_{p(\xi)} - \zeta_{P(\xi)})$ is a non-zero 
perverse subsheaf of the irreducible sheaf $P(\xi)$.
Hence, $\sigma_{P(\xi)} = \zeta_{P(\xi)}$.   

The next step in the extension of $\eta_{S}^{mic, Q_{H}}$ is the requisite
extension of $\tau_{S}^{mic}(P(\xi))$ to $A_{S}^{mic,Q_{H}}  \times
\langle \sigma \rangle$.  The ${^\vee}H$-orbit $S$  specifies special
stalks of the local system $Q^{mic}(P(\xi))$ (Lemma 24.3 \cite{abv}).
We choose one, which we cryptically denote by 
$Q^{mic}(P(\xi))_{x,\nu}$ ($x$ belongs to $S$, see Definition 24.7 \cite{abv}).
Recall that the local system $Q^{mic}(P(\xi))$ corresponds to the representation
$\tau_{S}^{mic}(P(\xi))$.  
It is a  representation  of $A_{S}^{mic, Q_{H}}$ on $Q^{mic}(P(\xi))_{x,\nu}$.
The automorphism $\sigma_{P(\xi)}$  induces an automorphism of
$Q^{mic}(P(\xi))_{x,\nu}$  ((25.1)(h), (24.10)(b) and
Definition 24.11 \cite{abv}).  We denote this automorphism by
$\tau_{S}^{mic}(P(\xi))(\sigma_{P(\xi)})$.  This notation is consistent with
(25.1)(j) \cite{abv} and is appropriate, for
$\tau_{S}^{mic}(P(\xi))(\sigma_{P(\xi)})$  is an operator
intertwining  $\tau_{S}^{mic}(P(\xi))$ with $\tau_{S}^{mic}(P(\xi))
\circ \sigma = \tau_{S}^{mic}(P(\xi))$.  It is a  consequence of
$\sigma_{P(\xi)} = \zeta_{P(\xi)}$ that 
$\tau_{S}^{mic}(P(\xi))(\sigma_{P(\xi)})$
is multiplication by  $\zeta \in
U_{m}$.  This yields the desired extension of $\tau_{S}^{mic}(P(\xi))$ to
$A_{S}^{mic,Q_{H}} 
\times \langle \sigma \rangle$.

In completing our definition of $\eta_{S}^{mic,Q_{H}}(\sigma)$ we
would like to arrange things so that some
pairing with $K \mathcal{P}(X({^\vee}H^{\Gamma}),
{^\vee}H^{Q_{H}}, \sigma)$ results in a formula akin to
(\ref{etapair}).  The natural objects to pair with $(P(\xi),
\sigma_{P(\xi)})$ are objects of the form $(\pi(\xi'), \zeta')$, where $\zeta'
\in U_{m}$ is to be interpreted as a self-intertwining operator of
(any representative of) $\pi(\xi')$.  The  pairs $(P(\xi),
\sigma_{P(\xi)})$ and $(\pi(\xi'), \zeta')$ may evidently be thought of as
objects in
$$K\mathcal{P}(X({^\vee}H^{\Gamma}), 
{^\vee}H^{Q_{H}}, \sigma) = K \mathcal{P}(X({^\vee}H^{\Gamma}),
{^\vee}H^{Q_{H}}) \otimes \mathbb{Z}[U_{m}]$$
and $K\Pi(H/\mathbb{R})_{\hat{Q}_{H}} \otimes \mathbb{Z}[U_{m}]$   respectively.
We extend pairing  
(\ref{pp1}) to a pairing
$$K\Pi(H/\mathbb{R})_{\hat{Q}_{H}} \otimes \mathbb{Z}[U_{m}] \times K
\mathcal{P}(X({^\vee}H^{\Gamma}), 
{^\vee}H^{Q_{H}}) \otimes \mathbb{Z}[U_{m}] \rightarrow \mathbb{C}$$
by defining
\begin{equation}
  \label{pp2}
\langle \pi(\xi') \otimes \zeta', P(\xi)\otimes  \zeta \rangle =
\zeta'\zeta \  \langle \pi(\xi'), P(\xi) \rangle.
\end{equation}

We finally extend (\ref{etadef}) to $h\sigma \in A_{S}^{mic, Q_{H}} \times
\langle \sigma \rangle$ by defining
\begin{equation}
  \label{etadef1}
  \eta_{S}^{mic,Q_{H}}(h\sigma) \nonumber  = \sum_{\xi \in
    \Xi({^\vee}H)^{Q_{H}}}  e(\xi)
(-1)^{\dim S_{\xi} - \dim S }\  \mathrm{tr}\left(
  \tau_{S}^{mic}(P(\xi))(h)\right)   (\pi(\xi) \otimes 1).
\end{equation}
Here, each $\pi(\xi) \otimes 1$ may be regarded as an irreducible
character, or equivalently as an irreducible character 
twisted by the identity self-intertwining operator
 (6 \cite{artreal}).  Viewed in this way,
twisted characters are paired with 
twisted perverse sheaves in  (\ref{pp2}), and  we obtain
\begin{align}
  \label{etapair1}
  \nonumber
  \langle \eta_{S}^{mic,Q_{H}}(h\sigma), (P(\xi),\sigma_{P(\xi)})
\rangle &=
(-1)^{\dim S} \, \zeta \ \mathrm{tr}\left(
  \tau_{S}^{mic}(P(\xi'))(h)\right) \\
& = (-1)^{\dim S}
\, \mathrm{tr}\left(
  \tau_{S}^{mic}(P(\xi'))(h\sigma_{P(\xi)})\right)
\end{align}
when $\sigma_{P(\xi)}$ is identified with $\zeta \in U_{m}$ as
above. 

The reader has likely noticed that the introduction of twisted objects here for
endoscopic groups appears artificial in the same way as it was in
(\ref{standend}).  The
introduction of twisted objects is justified  after comparison
with twisting for ${^\vee}G^{\Gamma}$.

We now consider twisting  for ${^\vee}G^{\Gamma}$. Under certain
circumstances one may imitate the procedure leading up to 
(\ref{etadef1}) for
${^\vee}G^{\Gamma}$.   Let us first assume that $S \subset
X({^\vee}H^{\Gamma})$ is an  ${^\vee}H$-orbit and
\begin{equation}
  \label{srel}
  S' = X(\epsilon)(S) \subset X({^\vee}G^{\Gamma})
  \end{equation}
is a ${^\vee}G$-orbit  (see (\ref{xep})). By
definition,  $S$ is $\sigma$-stable,  and so $S'$ is $\sigma$-stable.
Suppose $\xi \in \Xi({^\vee}G^{\Gamma})^{Q}$ is fixed by $\sigma$ and
that 
$$(P(\xi), 
\sigma_{P(\xi)}) \in \mathcal{P}(X({^\vee}G^{\Gamma}),
      {^\vee}G^{Q}, \sigma).$$
 The ${^\vee}H$-orbit $S$ again specifies 
stalks of $Q^{mic}(P(\xi))$ and we fix one, $Q^{mic}(P(\xi))_{x,\nu}$
($x$ belongs to $S \subset S'$, see 
Definition 24.7 and 
Lemma 25.3 \cite{abv}).
As before we denote the  representation
of $A_{S'}^{mic, Q}$ on $Q^{mic}(P(\xi))_{x,\nu}$ by $\tau_{S'}^{mic}(P(\xi))$.
The automorphism $\sigma_{P(\xi)}$  induces an automorphism
$\tau_{S'}^{mic}(P(\xi))(\sigma_{P(\xi)})$ of
$Q^{mic}(P(\xi))_{x,\nu}$ (\emph{cf.} (25.1)(h) and (25.1)(i)
\cite{abv}).
It  is  an operator 
intertwining  $\tau_{S'}^{mic}(P(\xi))$ with $\tau_{S'}^{mic}(P(\xi)) \circ \sigma$,
and yields an extension of  $\tau_{S'}^{mic}(P(\xi))$ to $A_{S'}^{mic,Q}
\rtimes \langle \sigma \rangle$.
Unlike the construction for ${^\vee}H^{\Gamma}$, the intertwining
operators need not be scalars.

The next step towards the construction of $\eta_{S'}^{mic,Q}(\sigma )$ is the
definition of commensurate objects to pair with $(P(\xi), \sigma_{P(\xi)})$.
The pairs $(\pi(\xi), \zeta) = \pi(\xi) \otimes \zeta$ above and the
equivariance of Theorem 
\ref{equivthm}  hint that these
objects should be (equivalence classes of) $\vartheta^{\Gamma}$-stable
representations of strong 
real forms together with their intertwining operators.  The $m$th
roots $\zeta$ in the pairs
$(\pi(\xi),\zeta)$ also indicate that the intertwining operators should
in some way ``remember'' the order $m$ of $\sigma$.
\begin{lem}
  \label{autm}
Suppose $\delta \in G^{\Gamma} - G$ is a strong real form and $g \in
G$ such that $\vartheta^{\Gamma} \circ \mathrm{Int}(g)(\delta) =
\delta$, \emph{i.e.} $\delta$ is equivalent to
$(\vartheta^{\Gamma})^{-1}(\delta)$.  Then $(\vartheta^{\Gamma} \circ \mathrm{Int}(g))^{m} =
\mathrm{Int}\left((\vartheta + \vartheta^{2} + \cdots +
\vartheta^{m})(g)\right)$ and $(\vartheta + \cdots +
\vartheta^{m})(g)$ belongs to $G(\mathbb{R}, \delta)$.
\end{lem}
\begin{proof}
  Since
  $$\sigma^{m} = (\mathrm{Int}(s) \circ {^\vee}\vartheta^{\Gamma})^{m} =
  \mathrm{Int}\left((1 + {^\vee}\vartheta + \cdots + 
{^\vee}\vartheta^{m-1})(s)\right) \circ
({^\vee}\vartheta^{\Gamma})^{m}$$ 
is the identity and $({^\vee}\vartheta^{\Gamma})^{m}_{|{^\vee}G} =
{^\vee}\vartheta^{m}$ is 
distinguished, it follows that ${^\vee}\vartheta^{m}$ is
also the identity (16.5 \cite{humphreys}).  The compatibility of
${^\vee}\vartheta$ with the distinguished automorphism
$\vartheta$ forces the two automorphisms to have the same
order.  Therefore $\vartheta^{m}$ is the identity and the first
assertion follows.  The second assertion follows directly from the
definition of $G(\mathbb{R}, \delta)$ and 
$$\delta = (\vartheta^{\Gamma} \circ \mathrm{Int}(g))^{m} (\delta) =
 \mathrm{Int}\left(( \vartheta + \cdots +
\vartheta^{m})(g)\right)(\delta)$$
\end{proof}

\begin{definition}
  \label{twistedrep}
A \emph{twisted representation of a strong real form} for $(G^{\Gamma},
\vartheta^{\Gamma}, \sigma)$
is a pair $((\pi, \delta), \mathcal{I})$  satisfying the following conditions.
\begin{enumerate}
\item  $(\pi, \delta)$ is a
representation of a strong real form of $G^{\Gamma}$ (Section
\ref{recall}).
\item $\vartheta^{\Gamma} \cdot (\pi, \delta) $ is equivalent to $(\pi,
\delta)$, \emph{i.e.} there exists $g \in G$ such that
$(\vartheta^{\Gamma}) \circ \mathrm{Int}(g) (\delta) = \delta$ and
$\pi \circ \vartheta \circ \mathrm{Int}(g)$ is
infinitesimally equivalent to $\pi$ (see (\ref{repact})). 
\item $\mathcal{I}$ is  a linear automorphism of the representation space of
  $\pi$ such that for some $y \in G(\mathbb{R}, \delta)$ the operator
  $\mathcal{I} \pi(y)$ exhibits the equivalence of part 2
   above, \emph{i.e.} 
  $$\pi \circ \vartheta \circ \mathrm{Int}(g)   = \mathcal{I}\pi(y) \, \pi \, (
  \mathcal{I}\pi(y))^{-1} $$
  (on the appropriate dense subspace).
  \item $(\mathcal{I} \pi(y))^{m} = \pi\left((\vartheta + \cdots +
    \vartheta^{m})(g)\right)$.   
\end{enumerate}
\end{definition}
We note any other element of $G$ producing the same equivalence in
part 3 of 
this definition is of the form $gx$ where $x \in G(\mathbb{R},
\delta)$, and for this element $\mathcal{I}\pi(yx)$ is an intertwining
operator. The intertwining property then implies that
$$( \mathcal{I}\pi(yx))^{m} = ( \mathcal{I} \pi(y) \pi(x))^{m} =
\pi(\left(\vartheta \circ 
\mathrm{Int}(g) + \cdots +(\vartheta \circ
\mathrm{Int}(g))^{m} \right)(x)) \, (\mathcal{I}\pi(y))^{m}$$
which may be verified to equal
$$\pi(\left(\vartheta \circ
\mathrm{Int}(g) + \cdots +(\vartheta \circ
\mathrm{Int}(g))^{m} \right)(x)) \, \pi\left((\vartheta + \cdots +
    \vartheta^{m})(g)\right) =  \pi\left((\vartheta + \cdots +
    \vartheta^{m})(gx)\right)$$
    using Lemma \ref{autm}.
Thus, Definition \ref{twistedrep} is well-defined.
\begin{definition}
  \label{twistedrep1}
Two twisted representations of strong real forms $((\pi, \delta), \mathcal{I})$
and $((\pi', \delta'), \mathcal{I}')$ as in Definition \ref{twistedrep} are
\emph{equivalent} if 
\begin{enumerate}
\item $(\pi,\delta)$ is equivalent to $(\pi',\delta')$, \emph{i.e.}
  there exists $g_{0} \in G$ such that $g_{0}\delta g_{0}^{-1} =
  \delta'$ and $\pi \circ \mathrm{Int}(g_{0}^{-1}) = A\pi' A^{-1}$ for
  some intertwining operator $A$.

\item $$A^{-1} (\mathcal{I}\pi(y))^{-1} A \,\mathcal{I}'\pi'(y') = \pi'(g_{0}g^{-1}
  \vartheta^{-1}(g_{0}^{-1}) g')$$
for $y,y',g,g' \in G$ as in 2 Definition  \ref{twistedrep}.
\end{enumerate}
We denote the set of equivalence classes of  irreducible twisted representations
of strong real forms for $(G^{\Gamma}, \vartheta, \sigma)$ by $\Pi(G/
\mathbb{R}, \vartheta, \sigma)$.
\end{definition}
The last equation in  Definition \ref{twistedrep1} is motivated by the
fact that the operator on the left 
intertwines $\pi'$ with its composition under $\mathrm{Int}(g_{0}g^{-1}
\vartheta^{-1}(g_{0}^{-1}) g')$.  We omit the unpleasant verification
that $g_{0}g^{-1} \vartheta^{-1}(g_{0}^{-1}) g'$ lies in
$G(\mathbb{R}, \delta')$ and that Definition 5.11 is well-defined.
Definitions \ref{twistedrep} and \ref{twistedrep1} have
analogues for canonical projective representations of strong real
forms (Definition 10.3 \cite{abv}).  We omit the details once again
and write $\Pi^{z}(G/  \mathbb{R}, \vartheta, \sigma)$ for the
equivalence classes of type $z$.  The subset of type $\hat{Q}$
is denoted by  $\Pi^{z}(G/\mathbb{R}, \vartheta, \sigma)_{\hat{Q}}$.

The elements of $\Pi^{z}(G/\mathbb{R}, \vartheta, \sigma)_{\hat{Q}}$
are the objects we wish to pair with $\mathcal{P}(X({^\vee}G^{\Gamma}),
{^\vee}G^{Q}, \sigma)$.  This presents us with the following problem:
given $\sigma$-fixed $\xi',\xi \in \Xi({^\vee}G^{\Gamma})^{Q}$ and
(representatives) 
$(\pi(\xi'), \mathcal{I})$ , $(P(\xi), \sigma_{P(\xi)})$ in
$\Pi^{z}(G/\mathbb{R}, \vartheta, \sigma)_{\hat{Q}}$ and
$\mathcal{P}(X({^\vee}G^{\Gamma}), {^\vee}G^{Q}, \sigma)$
respectively, how does one define
\begin{equation}
  \label{prepp3}
\langle (\pi(\xi'), \mathcal{I}) , (P(\xi), \sigma_{P(\xi)}) \rangle
\in \mathbb{C}
\end{equation}
in a canonical fashion.
For endoscopic groups, the pairing (\ref{pp2}) was  defined
using the canonical base point $(\pi(\xi), 1)$ with trivial
self-intertwining operator, and the canonical base point $(P(\xi),
\sigma_{P(\xi)})$ with $\sigma_{P(\xi)}$ acting trivially.  If one
could similarly choose a canonical intertwining operator
$\mathcal{I}_{\xi'}$ and a 
canonical automorphism $\sigma_{\xi}$ in (\ref{prepp3}), then all
other relevant intertwining operators and automorphisms would differ from
these by roots of unity in $U_{m}$ and we could set
\begin{equation}
  \label{pp3}
\langle (\pi(\xi'), \zeta' \mathcal{I}_{\xi'}) , (P(\xi),
\zeta\sigma_{\xi }) \rangle
 = \zeta' \zeta \, \langle \pi(\xi'), P(\xi) \rangle, \ \zeta', \zeta \in U_{m}.
\end{equation}
This would constitute a canonical pairing
$$K \Pi^{z}(G/\mathbb{R},
\vartheta, \sigma)_{\hat{Q}} \times K \mathcal{P}(X({^\vee}G^{\Gamma}),
         {^\vee}G^{Q}, \sigma) \rightarrow \mathbb{C}.$$

We are not prepared to solve the general problem of choosing a canonical
intertwining operator $\mathcal{I}_{\xi'}$ for $\pi(\xi')$ as indicated.   It is
solved in the examples of twisted endoscopy for $G = \mathrm{GL}_{n}$,
where $Q$ is trivial, in 2.1 \cite{arthurbook}, 8 \cite{amr} and 3.2
\cite{mok}.  In these examples one only encounters the quasisplit
forms of general linear groups, and therefore has recourse to Whittaker data to
normalize intertwining operators.  Perhaps there is a
simple resolution of the problem for arbitrary quasisplit groups, but
this is less likely for all strong inner forms.

The problem of choosing a canonical automorphism $\sigma_{\xi}$
for $P(\xi)$ does have a simple resolution using the map $\chi$ as in
(\ref{chimap}).  It is known that the irreducible constructible sheaf
$\mu(\xi)$ occurs in $\chi(P(\xi))$ 
with multiplicity one ((7.11)(e) \cite{abv}).  In consequence, an
automorphism $\chi(\sigma_{P(\xi)})$ restricts to an automorphism of
$\mu(\xi)$.  The latter automorphism  induces a self-intertwining operator
of the character $\tau^{loc}_{S_{\xi}}(\mu(\xi))$ as in the discussion preceding
(\ref{epstar1}).  As such, it is a root of unity in $U_{m}$.  We choose
the canonical automorphism $\sigma_{\xi}$ of $P(\xi)$ to be
the one for which this root of unity is $1 \in U_{m}$.  

At last we are able to describe twisted endoscopic lifting, at
least under the assumption  that the canonical intertwining operators in the
definition of  (\ref{prepp3}) are defined.  Under this assumption
we define
\begin{equation}
  \label{etadef2}
  \eta_{S'}^{mic,Q}(\sigma ) =  \sum_{\xi \in \Xi({^\vee}G)^{Q}} e(\xi)
(-1)^{\dim S_{\xi} - \dim S' }\ \mathrm{tr}\left(
  \tau_{S'}^{mic}(P(\xi))(\sigma_{\xi })\right) \,  (\pi(\xi),
  \mathcal{I}_{\xi}) ,
\end{equation}
where $\tau_{S'}^{mic}(P(\xi))(\sigma_{\xi}) = 0$ when $\sigma \cdot \xi
\neq \xi$.  This is to be regarded as a complex combination of twisted
characters.  
Using pairing (\ref{prepp3}), one may compute as in (\ref{etapair1})
that for $(P(\xi), 
\sigma_{P(\xi)}) \in K \mathcal{P}(X({^\vee}G^{\Gamma}),
      {^\vee}G^{Q}, \sigma)$ we have
$$\langle \eta_{S'}^{mic,Q}, (P(\xi),\sigma_{P(\xi)}) \rangle =
  (-1)^{\dim S'} \, \mathrm{tr}\left(
  \tau_{S'}^{mic}(P(\xi))(\sigma_{P(\xi)}) \right).$$

The twisted endoscopic lifting identity of Theorem 25.8 \cite{abv}
may now be read as 
\begin{equation}
\label{endolift1}
\langle  \eta_{S'}^{mic,Q}(\sigma), (P(\xi),
\sigma_{P(\xi)}) \rangle = \langle \eta_{S}^{mic,Q_{H}}(\sigma),
\epsilon^{*}(P(\xi), \sigma_{P(\xi)}) \rangle
\end{equation}
for all $(P(\xi), \sigma_{P(\xi)}) \in K
\mathcal{P}(X({^\vee}G^{\Gamma}), {^\vee}G^{Q}, \sigma)$  (see (25.1)(j)
and (24.10) \cite{abv}).    This identity defines the twisted
endoscopic lifting map $\epsilon_{*}$ from the complex vector space
generated by $\eta_{S}^{mic,Q_{H}}(\sigma)$ to the vector space
generated by $\eta_{S'}^{mic,Q}$ via $\epsilon_{*} 
\eta_{S}^{mic,Q_{H}}(\sigma) = \eta_{S'}^{mic, Q}(\sigma)$
(\emph{cf.} (\ref{endolift})).

\section{Twisted endoscopy for $\mathrm{GL}_{N}$ and Arthur packets of
  classical groups}
\label{artsec}

In this concluding
section we  illustrate how the twisted endoscopic identity
(\ref{endolift1}) is applicable to Arthur packets as defined in
\cite{abv}.  We shall  do this in  the framework of
Examples \ref{artexample} and \ref{artexample1} which were chosen to match
the endoscopic classification of Arthur (\cite{arthurbook}).

We begin by recalling the definitions leading up to Arthur packets.  In
these definitions ${^\vee}G^{\Gamma}$ is an arbitrary weak E-group.  
An \emph{Arthur 
  parameter} (or \emph{A-parameter}) is a homomorphism
$$\psi: W_{\mathbb{R}} \times \mathrm{SL}_{2} \rightarrow
     {^\vee}G^{\Gamma}$$
such that $\psi_{|W_{\mathbb{R}}}$ is a tempered  (\emph{i.e.} bounded)
L-parameter and $\psi_{|\mathrm{SL}_{2}}$ is holomorphic.

Let  $\psi$ be an A-parameter. There is an \emph{associated L-parameter}
$$\phi_{\psi}:  W_{\mathbb{R}} \rightarrow {^\vee}G^{\Gamma}$$
defined by
$$\phi_{\psi}(w) = \psi\left( w, \begin{bmatrix} |w|^{1/2} & 0 \\ 0
  & |w|^{-1/2} \end{bmatrix} \right), \ w \in W_{\mathbb{R}}.$$
There is also an associated \emph{Arthur component group}
$A_{\psi}$, which is the component group of the centralizer in
$^{\vee}G$ of the image of $\psi$.   These definitions are due to
Arthur (\S 4 \cite{arthur89}).

The first connection with the theory of
\cite{abv} is that the Arthur component group
$A_{\psi}$ is equal to an equivariant micro-fundamental group
$A^{mic, Q}_{S_{\psi}}$. Here, $Q$
is trivial and $S_{\psi} \subset 
X({^\vee}G^{\Gamma})$ is the ${^\vee}G$-orbit\footnote{The orbit
  $S_{\psi}$ here is not to be  confused with the centralizer
  $S_{\psi}$ in \cite{arthur89} or \cite{arthurbook}.} corresponding to 
$\phi_{\psi}$ via Proposition 6.17 \cite{abv} (see also Definition
24.15 and Proposition 22.9 \cite{abv}). 
For this reason we assume from now on that $Q$ is trivial. 

Recall from the previous section that for every $\xi \in
\Xi({^\vee}G^{\Gamma})^{Q}$ there is a representation
$\tau^{mic}_{S_{\psi}}(P(\xi))$ of $A_{S_{\psi}}^{mic,Q} = A_{\psi}$. Suppose that
the hypotheses of Theorem \ref{equivthm} are satisfied so that
(\ref{llc3}) holds.  Then
the \emph{Arthur packet} (or
\emph{A-packet}) of $\psi$  \cite{abv} may be
described as the 
set of $\pi(\xi) \in \Pi^{z}(G/\mathbb{R})_{\hat{Q}}$ whose
multiplicity in
$$\eta_{\psi} = \eta_{S_{\psi}}^{mic,Q} = \sum_{\xi \in
  \Xi({^\vee}G^{\Gamma})^{Q}} e(\xi) (-1)^{\dim S_{\xi} - \dim S_{\psi} }\ 
\dim \tau_{S}^{mic}(P(\xi))\  \pi(\xi)$$
is positive (\emph{cf.} Definition 19.13, Definition 19.15, Corollary 19.16,
Definition 22.6, Corollary 24.9 (a) \cite{abv}).  In other words, the
A-packet of $\psi$ is
\begin{equation}
  \label{apacket}
\Pi^{z}(G/\mathbb{R})_{\hat{Q},\psi} = \{\pi(\xi) :
\tau_{S_{\psi}}^{mic}(P(\xi)) \neq 0\}.
\end{equation}
We should emphasize that  the assumption of $Q$ being trivial in our
description results in A-packets which are potentially smaller than
the extended A-packets defined for $Q =
\pi_{1}({^\vee}G)^{alg}$ in Definition 22.6 \cite{abv}.  To
be honest, Adams-Barbasch-Vogan do not even give a name to the sets in
(\ref{apacket}), although they play a key role in endoscopy
(\emph{cf.} Definition 26.8 \cite{abv}).

Henceforth, we work under the assumptions of Examples \ref{artexample}
and \ref{artexample1}.     In particular, $G = \mathrm{GL}_{N}$ and
the endoscopic group $H$ is a product of symplectic or orthogonal groups.  
Our task is to reveal the relationship between the twisted endoscopic
transfer identity (\ref{endolift1}) and  the A-packets
we have just defined.

We fix  an A-parameter $\psi_{H}: W_{\mathbb{R}} \times
\mathrm{SL}_{2}  \rightarrow {^\vee}H^{\Gamma}$ and define
$\eta_{\psi_{H}} = 
\eta_{S_{\psi_{H}}}^{mic,Q_{H}}$.  Recall that $S_{\psi_{H}}$ is the
${^\vee}H$-orbit corresponding to $\phi_{\psi_{H}}$ via Proposition
6.17 \cite{abv}, and $Q_{H}$ is
trivial.  More generally, we set $\eta_{\psi_{H}}(h\sigma) =
\eta_{S_{\psi_{H}}}^{mic,Q_{H}}(h\sigma)$ as in (\ref{etadef1}).  The
virtual representation $\eta_{\psi_{H}}(\sigma)$ is to appear on the
right-hand side of (\ref{endolift1}).

For the left-hand side,  set $\psi_{G} = \epsilon \circ \psi_{H}$ so
that  $\psi_{G}$ is an A-parameter for $G$. It is a simple exercise to
show that the ${^\vee}G$-orbit
$S_{\psi_{G}}$ corresponding to $\phi_{\psi_{G}} = \epsilon \circ
\phi_{\psi_{H}}$ is equal to the ${^\vee}G$-orbit of
$X(\epsilon)(S_{\psi_{H}})$ (\emph{cf.} (\ref{srel}), (\ref{xep})).
Let $\eta_{\psi_{G}}(\sigma) = \eta_{S_{\psi_{G}}}^{mic,Q}(\sigma)$
as in (\ref{etadef2}).  Recall that in this definition we are making
use of a canonical normalization of intertwining operators (2.1
\cite{arthurbook}, 8 \cite{amr}).

The twisted endoscopic transfer identity (\ref{endolift1}) is now in the form
\begin{equation}
\label{endolift2}
\langle  \eta_{\psi_{G}}(\sigma), (P(\xi),
\sigma_{P(\xi)}) \rangle = \langle \eta_{\psi_{H}}(\sigma),
\epsilon^{*}(P(\xi), \sigma_{P(\xi)}) \rangle.
\end{equation}
This identity is a sheaf-theoretic analogue of Arthur's equation
(2.2.3) \cite{arthurbook} in which  $\eta_{\psi_{H}}(\sigma)$ plays
the role of the stable linear form defining an A-packet.

Equation
(\ref{endolift2}) is significant in its own right, for it yields
information about the elusive $\eta_{\psi_{H}}(\sigma)$ in terms of
$\eta_{\psi_{G}}(\sigma)$, which is better understood.  Let us investigate
a specific example.  The reader should revisit Examples
\ref{artexample} and \ref{artexample1} for the background.

\begin{exmp}
  \label{gl2}
Take $N = 2$, $G = \mathrm{GL}_{2}$, $N_{O} = 1$ and $N_{S}' = 1$.
Then $\sigma = {^\vee} \vartheta = \vartheta$ has order two, 
${^\vee}H^{\Gamma} = \mathrm{Sp}_{2} \times \Gamma  = \mathrm{SL}_{2}
\times \Gamma$, and $H = \mathrm{SO}_{3} \cong \mathrm{PGL}_{2}$.  The
second invariants of both ${^\vee}G^{\Gamma}$ and ${^\vee}H^{\Gamma}$
are trivial.  As is customary, we ignore the copy of $\Gamma$ in the
previous direct products. 
A-parameters $\psi_{H}: W_{\mathbb{R}}  \times \mathrm{SL}_{2}
\rightarrow \mathrm{SL}_{2} \times \Gamma$ may be divided into two
types:  those whose restriction $(\psi_{H})_{|\mathrm{SL}_{2}}$ is
trivial and those whose restriction $(\psi_{H})_{|\mathrm{SL}_{2}}$ is
the identity map.  The former case is the case of tempered
A-parameters.  In this case the A-packets reduce to L-packets and both
$\eta_{\psi_{H}}$ and $\eta_{\psi_{G}}$ are well-known for any
groups (page 19 \cite{abv}).

In the latter case $\psi_{H}(W_{\mathbb{R}})$ is central in
$\mathrm{SL}_{2}$ and it is harmless to assume that it is trivial. 
One may compute that 
$$\phi_{\psi_{H}}(z) = \phi_{\psi_{G}}(z) = \begin{bmatrix} |z\bar{z}|^{1/2} & 0 \\
  0 & |z \bar{z}|^{-1/2} \end{bmatrix}, \ z \in \mathbb{C}^{\times}$$
and $\phi_{\psi_{H}}(j) =\phi_{\psi_{G}}(j)=I$.  The L-packet of
$\phi_{\psi_{G}}$ (of type $\hat{Q}$) is the trivial 
representation of $\mathrm{GL}(2,\mathbb{R})$.  Similarly the L-packet of
$\phi_{\psi_{H}}$ (of type $\hat{Q}_{H}$) is  the trivial
representation of $\mathrm{PGL}(2, \mathbb{R})$.

It is generally true that the A-packet of an A-parameter contains the
L-packet of its 
associated L-parameter (Lemma 19.14 (b) \cite{abv}).  By a comparison
with the A-packets of Arthur for general linear groups (pages 24-25
and (2.2.1) \cite{arthurbook}), one would expect  the A-packet of 
$\psi_{G}$ to be equal to the L-packet of $\phi_{\psi_{G}}$.  This
is indeed true, although we see no simple explanation for this fact,
even for $N=2$.  The explanation given in \cite{abv} relies on
$\psi_{G}$ being \emph{unipotent} (Definition 27.1, Theorem 27.18 (d),
Example 19.17 \cite{abv}).  The explanation is  equally valid for $H =
\mathrm{PGL}_{2}$ and so we may write
$$\Pi(G/\mathbb{R})_{\hat{Q},\psi_{G}} = \{\mathbf{1}_{\mathrm{GL}(2,
  \mathbb{R})} \}, \ \Pi(H/\mathbb{R})_{\hat{Q}_{H},\psi_{H}} =
\{\mathbf{1}_{\mathrm{PGL}(2, 
  \mathbb{R})} \}.$$
Incidentally, the triviality of $Q$ and $Q_{H}$ manifests itself here
in the absence of the trivial representations of the other (strong) real
forms of $G$ and $H$.

Even though the A-packets in this example are transparent, it is still
instructive to consider (\ref{endolift2}).  On the left of
(\ref{endolift2}) we have 
$$\eta_{\psi_{G}}(\sigma) =  e(\xi) \ \mathrm{tr}\left(
  \tau_{S_{\psi_{G}}}^{mic}(P(\xi))(\sigma_{\xi })\right) \,  (\pi(\xi),
  \mathcal{I}_{\xi}).$$
As we know that the A-packet of $\psi_{G}$ is a singleton, the sum
over $\xi \in \Xi({^\vee}G^{\Gamma})^{Q}$ in (\ref{etadef2}) is reduced to a
single index $\xi = (S_{\psi_{G}}, \mathbf{1})$.  The ${^\vee}G$-orbit
$S_{\psi_{G}}$ has been described earlier.  The second term $\mathbf{1}$
denotes the trivial, and only, representation of the 
component group
$A_{S_{\psi_{G}}}^{loc,Q} = {^\vee}G_{\phi_{\psi_{G}}}/
({^\vee}G_{\phi_{\psi_{G}}})_{0} = \{1\}$.
It is an immediate consequence of  Definition 15.8 \cite{abv} that
$e(\xi) = 1$.   According to Corollary 24.9 (b) and Proposition 23.2
(b) \cite{abv}, the representation
$\tau_{S_{\psi_{G}}}^{mic}(P(\xi))$ is  equal to
$\tau_{S_{\psi_{G}}}^{loc}(\mu(\xi)) = \mathbf{1}$. The
intertwining operator 
$\tau_{S_{\psi_{G}}}^{mic}(P(\xi))(\sigma_{\xi })= \pm 1$ has been
normalized to equal $1$.  Finally, a straightforward verification  of 2.1
\cite{arthurbook} shows that the self-intertwining operator
$\mathcal{I}_{\xi}$ of $\pi(\xi) =
\mathbf{1}_{\mathrm{GL}(2,\mathbb{R})}$ is normalized to equal $1$.
In short, $\eta_{\psi_{G}}(\sigma) = (\mathbf{1}_{\mathrm{GL}(2,
  \mathbb{R})},1)$.

The pairing on the left-hand side of (\ref{endolift2}) is taken with
twisted perverse sheaves $(P(\xi'), \sigma_{\xi'})$, for any  $\xi' =
(S_{\xi'}, \tau_{\xi'}) 
\in \Xi(^{\vee}G^{\Gamma})^{Q}$.  Since the component groups
$A_{S_{\xi'}}^{loc}$ are all trivial, the representations $\tau_{\xi'}$
are all trivial.  This implies that the local system
$\mathcal{V}_{\xi'}$  on each ${^\vee}G$-orbit $S_{\xi'}$ is a
constant sheaf.

Let us describe the orbits $S_{\xi'}$.  The definition of the pairing
allows us to restrict our attention to the orbits contained in some variety
$X(\mathcal{O}, {^\vee}G^{\Gamma})$.  As it happens,
$\mathbf{1}_{\mathrm{GL}(2,\mathbb{R})}$ corresponds to
$$((y, \mathcal{F}(\lambda)), \tau) = \left( \left(\begin{bmatrix} -i &
  0 \\0 & i  \end{bmatrix},\  \mathcal{F} \left(\begin{bmatrix} 1/2 &
  0 \\0 & -1/2 \end{bmatrix} \right) \right), \mathbf{1} \right) $$
under the local Langlands Correspondence (\ref{llc1}).  Taking
$\mathcal{O} = {^\vee}G \cdot \lambda$, Proposition
6.16 \cite{abv} describes the $S_{\xi'}$ as  ${^\vee}G_{y}$-orbits of
the complete flag variety of $\mathrm{GL}_{2}$.  Clearly, the centralizer
${^\vee}G_{y}$ is 
the diagonal subgroup of $\mathrm{GL}_{2}$, and it is well-known that
the flag variety is isomorphic to the projective line
$\mathbb{P}^{1}$.  There are three resulting orbits in
$\mathbb{P}^{1}$, namely $\{0\}$, $\{ \infty \}$ and the open
orbit.  Let us label these orbits as $S_{1}$, $S_{2}$ and $S_{3}$
respectively.  It is safe to identify these orbits with
their corresponding ${^\vee}G$-orbits on  $X(\mathcal{O},
{^\vee}G^{\Gamma})$ (Proposition 7.14 \cite{abv}).  We may also
identify $S_{k}$ with $\xi_{k} = (S_{k}, \mathbf{1})$, for $k =
1,2,3$.  By construction $\pi(\xi_{1}) =
\mathbf{1}_{\mathrm{GL}(2,\mathbb{R})}$.  In consequence one may
compute the pairing (\ref{pp3})
\begin{equation}
  \label{lhsendolift}
\langle ( \eta_{\psi_{G}}(\sigma) , (P(\xi_{k}),
\sigma_{\xi_{k} }) \rangle
=\left\{ \begin{array}{ll} 1 & k = 1\\
  0 & k \neq 1 \end{array} \right.  
\end{equation}
(\emph{cf.} Theorem 1.24 \cite{abv}).  

The perverse sheaves $P(\xi_{k})$ are determined by the closure relations
among $S_{1}$, $S_{2}$ and $S_{3}$. Since
the closure $\overline{S_{\xi_{k}}}$ of the orbit  $S_{\xi_{k}}$ is
non-singular, the perverse sheaf $P(\xi_{k})$ is the constant
sheaf on  $\overline{S_{\xi_{k}}}$ concentrated in degree $-\dim S_{k}$
(Lemma 4.3.2 \cite{bbd}).

Moving to the right-hand side of (\ref{endolift2}), we observe that
there is an identical description of the relevant ${^\vee}H$-orbits of
$X({^\vee}H^{\Gamma})$ in terms of the three orbits of the complete flag
variety of ${^\vee}H = \mathrm{SL}_{2}$ (which is again isomorphic
to $\mathbb{P}^{1}$).  We denote the three orbits
by $S_{H,1}$, $S_{H,2}$, and $S_{H,3}$ to distinguish them from the
earlier ${^\vee}G$-orbits.  As earlier, the perverse sheaf $P(\xi_{H,k})$ is the
constant sheaf on  $\overline{S_{\xi_{H,k}}}$ concentrated in degree
$-\dim S_{H,k}$.  The restriction map $\epsilon^{*}$ is nearly
negligible and has the obvious
effect on the perverse sheaves: $\epsilon^{*}(P(\xi_{k}), \sigma_{\xi_{k}}) =
(P(\xi_{H,k}), \sigma_{\xi_{H,k}})$.  In consequence, identities
(\ref{endolift2}) and (\ref{lhsendolift}) combine to give us
\begin{equation}
  \label{rhsendolift}
\langle ( \eta_{\psi_{H}}(\sigma) , (P(\xi_{H,k}),
\sigma_{\xi_{H,k} }) \rangle
=\left\{ \begin{array}{ll} 1 & k = 1\\
  0 & k \neq 1 \end{array} \right.  .
\end{equation}
It follows in turn that $\eta_{\psi_{H}}(\sigma) = (\pi(\xi_{H,1}),1) =
(\mathbf{1}_{\mathrm{PGL}(2, \mathbb{R})},1)$ and
$\Pi(H/\mathbb{R})_{\hat{Q}_{H}, \psi_{H}} =
\{\mathbf{1}_{\mathrm{PGL}(2, \mathbb{R})}\} $ as expected. 
\end{exmp}

In Example \ref{gl2} we have delineated a method for computing
$\eta_{\psi_{H}}(\sigma)$, and thereby the A-packet
$\Pi^{z_{H}}(H/\mathbb{R})_{\hat{Q}_{H}, \psi_{H}}$, by using
(\ref{endolift2}).  It is possible to pursue this method for higher
rank $G = \mathrm{GL}_{N}$, but there are notable impediments.  We have already
mentioned the difficulty in proving that the A-packet
$\Pi(G/\mathbb{R})_{\hat{Q},\phi_{G}}$ is a singleton when $\psi_{G}$
is not unipotent.

Beyond this, one needs a description of the
${^\vee}G$-orbits and ${^\vee}H$-orbits occurring in the complete
local parameters. If the
infinitesimal character of the A-packet
$\Pi(G/\mathbb{R})_{\hat{Q}, \psi_{G}}$ is regular then good
descriptions for these orbits are given in terms of orbits on complete
flag varieties (\cite{yamamoto}, \cite{wyser}).
If the infinitesimal character is singular then additional work is
required to describe these orbits in terms of orbits on partial flag
varieties.

Finally, the perverse sheaves $P(\xi)$ are difficult to
characterize when the closure of the orbit $S_{\xi}$ is singular
(\cite{mcgovern}).  One 
approach to such a characterization is through the map $\chi$
of (\ref{chimap}) and its  relationship to representation theory via
the Kazhdan-Lusztig-Vogan algorithm ((7.11)(e) 
and Corollary 1.25 \cite{abv}).   Even if this approach does not provide an
overt conceptual characterization, it is amenable to computation (using
software such as the ATLAS of Lie groups and representations \cite{atlas}).

Unfortunately  this does not suffice, for we need a
characterization of \emph{twisted} perverse sheaves $(P(\xi),
\sigma_{\xi})$.  The obvious course of action is to use the extended map $\chi$
of (\ref{chimap1}) together with a twisted version of the
Kazhdan-Lusztig-Vogan algorithm (\cite{lusztigvogan}).  The details of
this strategy should follow \cite{av15} and 19 \cite{avvarxiv}.
One may describe the expected outcome of the strategy as follows.  The
map $\chi$ furnishes a decomposition of $(P(\xi), \sigma_{\xi})$ as a
virtual sum of irreducible twisted constructible sheaves.   Likewise, an
irreducible twisted representation may be decomposed as a virtual sum of
twisted standard representations.  The expected result is that there
is a simple combinatorial formula which exhibits an equivalence between the two
decompositions (\emph{cf.} Corollary 1.25 \cite{abv}).

This equivalence would allow one to recast the twisted endoscopic
identity (\ref{endolift2}), replacing the twisted perverse sheaves
with twisted constructible sheaves.  The twisted constructible sheaves
are simple to compute.  Beyond this, the  advantage of the
alternative twisted endoscopic identity would be the ease of comparison
with the twisted character identities defining the A-packets of Arthur
((8.3.4) \cite{arthurbook}).  The equality of the A-packets of
(\ref{apacket}) and the A-packets of \cite{arthurbook} promises to follow
from such a comparison.


\end{document}